\UseRawInputEncoding
\documentclass[12pt, reqno]{amsart}
\usepackage[margin=1in]{geometry}
\usepackage{amssymb,latexsym,amsmath,amscd,amsfonts}
\usepackage{latexsym}
\usepackage[mathscr]{eucal}
\usepackage{bm}
\usepackage{mathptmx}
\usepackage{amssymb}
\usepackage{amsthm}
\usepackage{dcolumn}
\usepackage[all]{xy}
\usepackage{enumitem}
\usepackage[utf8]{inputenc}

\def \qed {\hfill \vrule height6pt width 6pt depth 0pt}
\def\textmatrix#1&#2\\#3&#4\\{\bigl({#1 \atop #3}\ {#2 \atop #4}\bigr)}
\def\dispmatrix#1&#2\\#3&#4\\{\left({#1 \atop #3}\ {#2 \atop #4}\right)}
\newcommand{\beg}{\begin{equation}}
	\newcommand{\eeg}{\end{equation}}
\newcommand{\ben}{\begin{eqnarray*}}
	\newcommand{\een}{\end{eqnarray*}}

\newcommand{\C}{\mathbb C}
\newcommand{\R}{\mathbb R}

\newcommand{\T}{\mathbb T}

\newcommand{\E}{\mathbb E}
\newcommand{\He}{\mathbb H}
\newcommand{\Pe}{\mathbb P}
\newcommand{\D}{\mathbb D}
\newcommand{\G}{\mathbb G}

\newcommand{\Gg}{\mathbb G_2}

\newcommand{\lm}{\lambda}
\newcommand{\diag}{\text{diag}}
\newcommand{\penta}{\text{penta}}
\newcommand{\scal}{\text{scalar}}
\newcommand{\hexa}{\text{hexa}}
\newcommand{\hp}{{\text{HS}}^\perp}
\newtheorem{thm}{Theorem}[section]
\newtheorem{cor}[thm]{Corollary}
\newtheorem{lem}[thm]{Lemma}

\newtheorem{prop}[thm]{Proposition}
\numberwithin{equation}{section} \theoremstyle{definition}

\newtheorem{eg}[thm]{Example}

\def\textmatrix#1&#2\\#3&#4\\{\bigl({#1 \atop #3}\ {#2 \atop #4}\bigr)}
\def\dispmatrix#1&#2\\#3&#4\\{\left({#1 \atop #3}\ {#2 \atop #4}\right)}

\begin{document}
	
	\title{Rigidity of the structured singular value and applications} 
	\author{SOURAV PAL AND NITIN TOMAR}
	
	\address[Sourav Pal]{Mathematics Department, Indian Institute of Technology Bombay,
		Powai, Mumbai - 400076, India.} \email{sourav@math.iitb.ac.in}
	
	\address[Nitin Tomar]{Mathematics Department, Indian Institute of Technology Bombay, Powai, Mumbai-400076, India.} \email{tomarnitin414@gmail.com}



\begin{abstract}
The structured singular value $\mu_E$ for a linear subspace $E$ of $M_n(\mathbb C)$ is defined by 
\[
\mu_E(A)=1\slash\inf\{\|X\| \ : \ X \in E, \ \det(I_n-AX)=0 \}  \quad (A
\in M_n(\mathbb{C})),
\]
and $\mu_E(A)=0$ if there is no $X \in E$ with $\det(I_n-AX)=0$. It is well-known that $\mu_E(A)$ coincides with the spectral radius $r(A)$ when $E=\{cI_n: c \in \mathbb C \}$ and $\mu_E(A)=\|A\|$ when $E=M_n(\mathbb C)$, for all $A\in M_n(\mathbb C)$. Also, for any linear subspace $E$ satisfying $\{cI_n: c \in \mathbb C \} \subseteq E \subseteq M_n(\mathbb C)$, we have $r(A)\leq \mu_E(A) \leq \|A\|$. We prove that if $E=\{cI_n: c \in \mathbb C \}$ and $F$ is any linear subspace of $M_n(\mathbb C)$ containing $E$, then $\mu_E=\mu_F$ if and only if $E=F$. We prove the exact same rigidity theorem for the linear subspace consisting of the diagonal matrices of order $n$. On the contrary, when $E=M_n(\mathbb C)$, we show that there is a proper subspace $F$ of $M_n(\mathbb C)$, viz. the space of symmetric matrices such that $\mu_E=\mu_F=$ operator norm. Further, we characterize all linear subspaces $F\subseteq M_n(\mathbb C)$ such that $\mu_F$ coincides with the operator norm. Next, we show that in general there is no subspace $E$ of $M_n(\mathbb C)$ such that $\mu_E=$ the numerical radius, not even for $M_2(\mathbb C)$. Then, we prove that except for the spectral radius and operator norm, $\mu_E$ cannot be a convex combination of the spectral radius, numerical radius and operator norm for any $E\subseteq M_2(\mathbb C)$.
The unit ball of the function $\mu_E$ induces various domains in $\mathbb C^d$ depending on $E$, e.g., the symmetrized bidisc $\mathbb{G}_2$, the tetrablock $\mathbb{E}$, the pentablock $\mathbb{P}$ and the hexablock $\mathbb{H}$. The domains $\mathbb{G}_2$ and $\mathbb{E}$ correspond to the linear subspaces of scalar and diagonal matrices in $M_2(\mathbb C)$ respectively, whereas $\mathbb{P}$ and $\mathbb{H}$ arise from certain linear subspaces of upper triangular matrices in $M_2(\mathbb C)$. For each of these linear subspaces, we establish the rigidity of the corresponding structured singular value. 
\end{abstract}	
	
	\maketitle

\section{Introduction} 	
\noindent The structured singular value is one of the central notions in the theory of robust control, where it plays a key role in analyzing the stability of systems under structured uncertainties. Given a linear subspace $E$ of the space of complex matrices $M_n(\mathbb{C})$, the
\textit{structured singular value} $\mu_E: M_n(\C)\rightarrow \mathbb R$ is defined as
\begin{equation} \label{eqn:NEW-01}
\mu_E(A):=\frac{1}{\inf\{\|X\| \ : \ X \in E, \ \det(I_n-AX)=0 \}} \, ,  \quad A
\in M_n(\mathbb{C}).
\end{equation} 
In case there is no $X \in E$ with $\det(I_n-AX)=0$, $\mu_E(A)$ is defined to be equal to $0$. The linear subspace $E$ is referred to as the `structure'. The structured singular value can be defined in a more general setting of $m \times n$ matrices, e.g., see the last section of \cite{Abouhajar}. However, we restrict our attention to $M_n(\C)$ in this article. It is well-known that for all $A\in M_n(\C)$, $\mu_E(A)=r(A)$ when $E$ is the space of scalar matrices $\{cI_n:\, c\in \C  \}$ and $\mu_E(A)=\|A\|$ when $E=M_n(\C)$. Here, $r(A)$, $\|A\|$ are the spectral radius and operator norm of a matrix $A$ respectively and $I_n$ is the identity matrix in $M_n(\C)$. Also, for any subspace $E$ satisfying $\{ cI_n:\,c \in \C \} \subseteq E \subseteq M_n(\C)$, we have $r(A)\leq \mu_E(A) \leq \|A\|$ for any $A \in M_n(\C)$. See \cite{Packard} for proofs to these results. Most of the classical literature on structured singular value (e.g., see \cite{Doyle, DoyleII, Packard} and the references therein), especially in robust control theory, focuses on the structure $E$ comprising of block diagonal matrices whose diagonal blocks are either scalar matrices or diagonal matrices or unstructured $m \times m$ matrices that are referred to as full blocks. More precisely,
\[
E=\{\text{diag}[\delta_1I_{r_1}, \dotsc, \delta_kI_{r_k}, \Delta_{k+1}, \dotsc, \Delta_{k+\ell}]: \delta_i \in \C, \ \Delta_{k+j} \in M_{m_j}(\C), 1\leq i \leq k, 1 \leq j \leq \ell \},
\]
where $\diag[A_1, \dotsc, A_m]$ is the block-diagonal matrix with $A_j$ on the diagonal. Beyond its role in robust control, the structured singular value leads to an interesting interpolation problem which is known as the $\mu$-synthesis. For a linear subspace $E \subseteq M_n(\C)$, the goal of $\mu$-synthesis is the following: given distinct points $\lambda_1, \dotsc, \lambda_m$ in the open unit disc $\D$ and matrices $A_1, \dotsc, A_m \in M_n(\C)$, to find a necessary and sufficient condition such that there is an analytic map $F: \D \to M_n(\C)$ satisfying $F(\lm_j)=A_j$ for $1\leq j \leq m$, together with the constraint
$
\mu_E(F(\lm)) \leq 1$ for all $\lm \in \D$. As mentioned before, $\mu_E$ coincides with the operator norm when $E=M_n(\C)$ and the corresponding $\mu$-synthesis reduces to the classical interpolation into the norm unit ball of $M_n(\C)$. As is well-known, an interpolating function $F$ exists in this case if and only if the Pick matrix
\[
\begin{bmatrix}
(1-\overline{\lm}_i\lm_j)^{-1}(I-A_i^*A_j)
\end{bmatrix}_{i, j=1}^m
\]
is positive semi-definite, e.g., see Chapter X in \cite{Foias_Frazho} and Chapter XVIII in \cite{Ball}. Also, if $E=\{\alpha I_n : \alpha \in \C\}$, then the concerned $\mu_E$-synthesis becomes the spectral interpolation. A reader is referred to \cite{AglerI, Agler2004, Bercovici, Costara2005_II, Costara2005} for a detailed study of this case.

\smallskip 

The aim of this article is to study the rigidity of the structured singular value $\mu_E$ for different linear subspaces $E$ of $M_n(\C)$ and its interaction with a few domains in $\C^d$ that are in correspondence with the $\mu_E$-unit balls of some of those subspaces. In Theorem \ref{thm:sec2-01}, we prove that if $E$ is the subspace of scalar matrices in $M_n(\C)$, then $\mu_E$ is rigid in the sense that for any linear subspace $F\subseteq M_n(\C)$ containing $E$, $\mu_E=\mu_F$ if and only if $E=F$. We then consider the subspace of diagonal matrices and prove a similar rigidity result in Theorem \ref{thm:sec2-02}, which is to say that, if $E_{\diag} \subset M_n(\C)$ is the space of diagonal matrices and $\mathcal F$ is any subspace of $M_n(\C)$, then $\mu_{E_{\diag}}=\mu_{\mathcal F}$ if and only if $E_{\diag}=\mathcal F$. However, the same is not true when $E=M_n(\C)$. As we have mentioned earlier that $\mu_E(A)=\|A\|$ for all $A\in M_n(\C)$ when $E=M_n(\C)$. In Section \ref{sec_op}, we show the existence of a proper linear subspace of $M_n(\C)$ that induces the operator norm. Indeed, the space of all symmetric matrices is one such subspace. Then we characterize all linear subspace $E$ such that $\mu_E$ is equal to the operator norm. Next, we turn our attention to the numerical radius which is defined as
\[
w: M_n(\C) \to [0, \infty), \quad w(A):=\sup\{|x^*Ax| \ : \ x \in \C^n, \|x\|=1\}.
\] 
It is well-known that the numerical radius defines a norm on $M_n(\C)$, which is equivalent to the operator norm, and $r(A) \leq w(A) \leq \|A\|$ for all $A \in M_n(\C)$. Since the spectral radius and norm are induced by the scalar matrices and $M_n(\C)$ respectively, a natural question is triggered: for $n \geq 2$, does there exist a linear subspace $E \subseteq M_n(\C)$ such that $\mu_E(A)=w(A)$ for all $A \in M_n(\C)$? In Section \ref{sec-04-new}, we show that that the answer to this is not affirmative even for $n=2$. In fact, more is true. In the same section we prove that there is no linear subspace $E \subseteq M_2(\C)$ for which $\mu_E$ is a convex combination of the spectral radius, numerical radius and operator norm, except when $\mu_E$ coincides with either spectral radius or operator norm.

\smallskip

The $\mu_E$-unit balls corresponding to different linear subspace $E$ of $M_n(\C)$ give rise to different domains, see \cite{Hexablock} for a detailed discussion on this. For example, the symmetrized bidisc \cite{AglerYoung, AglerIII}, the tetrablock \cite{Abouhajar}, the pentablock \cite{AglerIV} and the hexablock \cite{Hexablock}, which are denoted by $\G_2$, $\mathbb E$, $\mathbb P$ and $\mathbb H$, respectively are induced by four different linear subspaces of $M_2(\C)$. Let us show how a domain is generated by a $\mu_E$-unit ball. Suppose $E$ is the linear subspace of $M_2(\C)$ consisting of diagonal matrices, i.e., $E=\left\{\begin{pmatrix}
z_1 & 0 \\
0 & z_2
\end{pmatrix} : z_1, z_2 \in \mathbb{C} \right\}$. It is shown in \cite{Abouhajar} that the tetrabolck \begin{equation}\label{eqn:NEW-02}
\mathbb{E}=\{(x_1, x_2, x_3) \in \mathbb{C}^3 :
1-x_1z_1-x_2z_2+x_3z_1z_2 \ne 0 \ \text{for all } \, z_1, z_2 \in \overline{\mathbb D}\}
\end{equation}
is originated in the associated $\mu_E$-unit ball. Indeed, for $X=\begin{pmatrix}
z_1 & 0\\
0& z_2
\end{pmatrix} \in E$ and $A=(a_{ij})_{i,j=1}^2 \in M_2(\mathbb{C})$, we have that
$\|X\|=\max\{|z_1|, |z_2|\}$ and $\det(I-AX)=1-a_{11}z_1-a_{22}z_2+\det(A)z_1z_2$. Thus, if we denote the $\mu_E$ of this case by $\mu_{\text{tetra}}$, then it follows from Equation-(\ref{eqn:NEW-01}) that
\begin{align*}
\mu_{\text{tetra}}(A)<1
& \iff \|X\|>1 \ \text{for all} \ X \in E \ \text{with} \ \det(I-AX)=0\\
& \iff \det(I-AX)\ne 0 \ \text{for all} \ X \in E \ \text{with} \ \|X\| \leq 1\\
& \iff 1-a_{11}z_1-a_{22}z_2+\det(A)z_1z_2\ne 0 \ \text{for all} \ z_1, z_2 \in \overline{\mathbb{D}}\\
& \iff (a_{11}, a_{22}, \det(A)) \in \mathbb{E}.
\end{align*}
Thus, $A=(a_{ij})_{i,j=1}^2$ is in the $\mu_E$-unit ball if and only if the point $(a_{11}, a_{22}, \det(A))$ belongs to the tetrablock $\mathbb E$ and consequently, one can write
\[
\E=\{\pi_\E(A): A \in M_2(\C), \ \|A\|<1\}=\{\pi_\E(A): A \in M_2(\C), \ \mu_{\text{tetra}}(A)<1\},
\]
where $\pi_\E(A)=(a_{11},a_{22},\det(A))$. It turns out that an interpolation problem between $\D$ and the $\mu_E$-unit ball is equivalent to a similar interpolation problem between $\D$ and the tetrablock $\mathbb E$,  e.g., see the last Section of \cite{Abouhajar}. Evidently, it is easier to deal with a bounded domain like $\mathbb E \subset \C^3$ than the norm-unbounded object, the $\mu_E$-unit ball in $\C^4$. 

Similarly, when $E$ consists of scalar matrices in $M_2(\C)$, i.e., $E=\left\{\begin{pmatrix}
z & 0 \\
0 & z
\end{pmatrix} : z \in \mathbb{C} \right\}$, the structured singular value $\mu_E$ coincides with the spectral radius and the $\mu_E$-unit ball leads to the symmetrized bidisc $\Gg$, which is given by
\[
\Gg=\{\Pi_2(A): A \in M_2(\C), \ \|A\|<1\}=\{\Pi_2(A): A \in M_2(\C), \ r(A)<1\},
\]
where $\Pi_{2}(A)=(\operatorname{tr}(A),\det(A))$. When $E=\left\{\begin{pmatrix}
z & w \\
0 & z
\end{pmatrix} : z, w \in \mathbb{C} \right\}$, the space of upper triangular matrices with same diagonals entries in $M_2(\C)$, the corresponding structured singular value $\mu_{\text{penta}}$ yields the pentablock $\Pe$ as follows:
\[
\Pe=\{\pi_{\Pe}(A): A \in M_2(\C), \ \|A\|<1\}=\{\pi_{\Pe}(A): A \in M_2(\C), \ \mu_{\text{penta}}(A)<1\},
\]
where $\pi_{\Pe}(A)=(a_{21}, \operatorname{tr}(A), \det(A))$ for $A=(a_{ij}) \in M_2(\C)$. When $E$ consists of all upper triangular matrices in $M_2(\C)$, then the associated structured singular value $\mu_{\text{hexa}}$ gives rise to the hexablock $\mathbb{H}$. To construct the hexablock, the authors of \cite{Hexablock} study the sets
\begin{equation*}
	\He_\mu=\{\pi_{\He}(A)  :  A \in M_2(\C), \  \mu_{\text{hexa}}(A)<1 \} \quad \text{and} \quad \He_N=\{\pi_{\He}(A) : A \in M_2(\C), \|A\|<1\},
\end{equation*}
where $\pi_{\He}: M_2(\C) \to \C^4, \ A \mapsto (a_{21}, a_{11}, a_{22}, \det(A))$ for $A=(a_{ij}) \in M_2(\C)$. The hexablock is obtained in the following way: 
\[
\He = int(\overline{\He}_\mu)=int(\widehat{\overline{\He}_N}), 
\]
where  $\widehat{\overline{\He}_N}$ is the polynomial convex hull of $\overline{\He}_N$. We refer to \cite{AglerYoung}-\cite{Agler2008} and \cite{Tirtha_Pal, Costara2004, Costara2005, NikolovIII, sourav14, pal-shalit, PZ} for a further reading on function theoretic, geometric and operator theoretic aspects of $\G_2$. The tetrablock $\E$ also has a rich literature, e.g. see \cite{Abouhajar, Alsalhi, Tirtha, Kosi, zwonek1, Young, Zwonek} and the references therein. An interested reader is directed to \cite{AglerIV, JindalII, KosinskiII, PalN, Guicong, GuicongII} for an elaborative discussion on the pentablock $\Pe$. More recently, the hexablock has been studied well in \cite{Hexablock, Pal_Hexa, Hexa_Su}. 

\smallskip 

In Section \ref{sec_rigid}, we prove the rigidity of the underlying structured singular values that induce the domains $\Gg, \E, \Pe$ and $\He$.

\subsection*{Notations.} Throughout the paper, we use the following notations:
\begin{align*}
	 (i) E_{\text{scalar}}&=\{zI_n: z \in \C\}, 
	\quad 	\qquad \qquad \qquad 
	(ii) \ \ E_{\text{diag}}=\{\text{diag}(z_1, \dotsc, z_n): z_1, \dotsc, z_n \in \C \}, \\
		(iii) \ E_{\text{penta}}&=\left\{ \begin{pmatrix}
			z & w \\
			0 & z
		\end{pmatrix} : z, w \in \C  \right\},
	\quad 	\qquad 
	(iv) \  E_{\text{hexa}}=\left\{ \begin{pmatrix}
		z_1 & w \\
		0 & z_2
	\end{pmatrix} : z_1, z_2, w \in \C  \right\}  \\
	(v) \ E_{\text{symm}}&=\{X \in M_n(\C): X=X^t\},
	\quad 	\qquad 
		(vi) \quad \  E_\theta=\left\{\begin{pmatrix}
		z_1 & w \\
		e^{i\theta} w & z_2
	\end{pmatrix} : z_1, z_2, w \in \C \right\},
\end{align*}
where $\theta \in \R$. The zero matrix of order $p \times q$ is denoted by $\Theta_{p \times q}$, and $\Theta_{p}=\Theta_{p\times p}$. For the identity matrix in $M_n(\C)$, we write $I_n$ or simply $I$. The standard orthonormal basis of $\C^n$ is denoted by $\{e_1, \dotsc, e_n\}$. Let $E_{ij} \in M_n(\C)$ denote the matrix with $(i,j)$-th entry $1$ and $0$ elsewhere.

	\section{Rigidity results for classical structured subspaces}\label{sec_rigid}

\noindent The main aim of this section is to establish the rigidity of structured singular value $\mu_E$ when $E \subseteq M_n(\C)$ is the subspace of scalar matrices or diagonal matrices, which is to say that for linear subspaces $E, E'$ of $M_n(\C)$ with $E_{\text{scalar}} \subseteq E, E'$, we prove that 
\begin{enumerate}
	\item  $\mu_E(A)=r(A)$ for all $A \in M_n(\C)$ if and only if $E=E_{\text{scalar}}$ ; \item $\mu_{E'}(A)=\mu_{E_\diag}(A)$ for all $A \in M_n(\C)$ if and only if $E'=E_\diag$.
\end{enumerate}
 Also, following the discussion of the previous section on the domains $\Gg, \E, \Pe$ and $\He$, we prove the rigidity of the associated structured singular values $r(\cdot), \mu_{\text{tetra}}, \mu_{\text{penta}},$ and $\mu_{\text{hexa}}$, respectively. More precisely, we study the following problems: to classify all linear subspaces $E \subseteq M_2(\C)$ containing the identity matrix for which, for every $A \in M_2(\C)$, one of the following holds:
\[
\mu_E(A)=r(A), \quad \mu_E(A)=\mu_{\text{tetra}}(A), \quad 
\mu_E(A)=\mu_{\text{penta}}(A), \quad \text{or} \quad 
\mu_E(A)=\mu_{\text{hexa}}(A).
\]
We provide a complete solution to these problems and show that each one admits a unique solution. We begin with the following result showing the structured singular values $r(\cdot), \mu_\diag, \mu_\penta$ and $\mu_\hexa$ are all distinct functions on $M_2(\C)$.

\begin{prop}\label{prop_mu}
	For $A \in M_2(\C)$, we have
	\[
	r(A) \leq \mu_{\text{diag}}(A) \leq \mu_{\text{hexa}}(A) \leq \|A\| \quad \text{and} \quad r(A) \leq  \mu_{\text{penta}}(A) \leq \mu_{\text{hexa}}(A) \leq \|A\|.
	\]
	Moreover, the functions $r(.), \ \mu_{\diag}, \ \mu_{\penta}, \ \mu_{\hexa}$ and $\|.\|$ are all distinct from each other. In particular, one can choose $A, B, C \in M_2(\C)$ such that 
	\begin{enumerate}
		\item $r(A)<\mu_{\penta}(A) < \mu_{\diag}(A)=\mu_{\hexa}(A)$; \smallskip 
		\item $r(B)=\mu_{\diag}(B) < \mu_{\penta}(B)=\mu_{\hexa}(B)$; \smallskip 
		\item $\mu_{\hexa}(C)<\|C\|$.\smallskip 
	\end{enumerate}
\end{prop}

\begin{proof}
	For linear subspaces $E_1, E_2$ satisfying $E_1 \subseteq E_2 \subseteq M_n(\C)$, we have $\mu_{E_1}(A) \leq \mu_{E_2}(A) \leq \|A\|$ for all $A \in M_n(\C)$. Since $E_{\text{scalar}} \subseteq E_{\text{diag}}, \ E_{\text{penta}} \subseteq E_{\text{hexa}} \subseteq M_2(\C)$ and $\mu_{E_{\text{scalar}}}(.)=r(.)$, we have
	\[
	r(A) \leq \mu_{\text{diag}}(A) \leq \mu_{\text{hexa}}(A) \leq \|A\| \qquad \text{and} \qquad r(A) \leq  \mu_{\text{penta}}(A) \leq \mu_{\text{hexa}}(A) \leq \|A\|
	\]
	for all $A \in M_2(\C)$. Let us define
	\[
	A=\begin{pmatrix}
		-1 & -1 \\
		1 & 1\\
	\end{pmatrix}, \quad B=\begin{pmatrix}
		0 & 0 \\
		1 & 0
	\end{pmatrix} \quad \text{and} \quad C=\begin{pmatrix}
		0 & 1 \\
		0 & 0
	\end{pmatrix}.
	\]
	It is not difficult to see that $\sigma(A)=\{0\}$ and so, $r(A)=0$. Let $D=\diag(d_1, d_2) \in E_{\diag}$. A simple calculation gives that $\det(I-AD)=1+d_1-d_2$ and so, we have
	\[
	\frac{1}{\mu_{\diag}(A)}=\inf_{d_1, d_2 \in \C}\|\diag(d_1, d_1+1)\| \leq \|\diag(-1\slash 2, 1\slash 2)\|=\frac{1}{2} \quad \text{and so,} \quad \mu_{\diag}(A) \geq 2.
	\]
	Since $\|A\|=2$, we have $\mu_{\diag}(A)=2$. For $X=\begin{pmatrix}
		z & w \\
		0 & z
	\end{pmatrix}$, it follows that $\det(I-AX)=1-w$. Then
	\[
	\frac{1}{\mu_{\penta}(A)}=\inf\left\{\left\|\begin{pmatrix}
		z & 1 \\
		0 & z
	\end{pmatrix}\right\|: z \in \C  \right\} \leq \left\|\begin{pmatrix}
		0 & 1 \\
		0 & 0
	\end{pmatrix}\right\|=1 \quad \text{and so,} \quad \mu_{\penta}(A) \geq 1.
	\]
	For $X=\begin{pmatrix}
		z & 1 \\
		0 & z
	\end{pmatrix}$, we have $\|X\| \geq 1$ since $\|Xe_2\|=\sqrt{1+|z|^2} \leq \|X\|$. Consequently, 
	\[
	\frac{1}{\mu_{\penta}}(A)=\inf\{\|X\|: X \in E_{\penta}, \det(I-AX)=0\} \geq 1 \quad \text{and so,} \quad \mu_{\penta}(A)=1.
	\]
	For $X=\begin{pmatrix}
		z_1 & w \\
		0 & z_2
	\end{pmatrix}$, a routine computation gives that $\det(I-AX)=1+z_1-z_2-w$. Then
	\[
	\frac{1}{\mu_{\hexa}(A)}=\inf_{z_1, z_2 \in \C} \left\|\begin{pmatrix}
		z_1 & 1+z_1-z_2 \\
		0 & z_2
	\end{pmatrix}\right\| \leq \left\|\begin{pmatrix}
		0 & 1\slash 2 \\
		0 & 1\slash 2
	\end{pmatrix}\right\|=\frac{1}{\sqrt{2}} \quad \text{and so,} \quad \mu_{\hexa}(A) \geq \sqrt{2}.
	\]
	Consequently, $r(A)<\mu_{\penta}(A) < \mu_{\diag}(A)=\mu_{\hexa}(A)=\|A\|=2$. Evidently,  $\|B\|=1, r(B)=0$ and $\det(I-BD)=1$ for all $D \in E_{\diag}$. Hence, $\mu_{\diag}(B)=0$. For $X=\begin{pmatrix}
		z & w \\
		0 & z
	\end{pmatrix}$, it is easy to see that $\det(I-BX)=1-w$. One can easily verify that
	\[
	\frac{1}{\mu_{\penta}(B)}=\inf\left\{\left\|\begin{pmatrix}
		z & 1 \\
		0 & z
	\end{pmatrix}\right\|: z \in \C  \right\}=1 \quad \text{and so,} \quad \mu_{\penta}(B) = 1=\|B\|.
	\]
	Since $\mu_{\penta}(B) \leq \mu_{\hexa}(B)\leq \|B\|$, we have $r(B)=\mu_{\diag}(B) < \mu_{\penta}(B)=\mu_{\hexa}(B)=\|B\|$. Clearly, $\|C\|=1$. Let $X=\begin{pmatrix}
		z_1 & w\\
		0 & z_2
	\end{pmatrix} \in E_{\hexa}$. Then $\det(I-CX)=1$ and so, $\mu_{\hexa}(C)=0$, which is strictly lesser that $\|C\|$. Consequently, $ \mu_{\hexa}(C)<\|C\|$, which gives the desired conclusion.
\end{proof}

We now present our first rigidity result associated with the spectral radius. We mention here that the rank one matrices play a crucial role in its proof. The authors of \cite{Talha} provided an explicit formula for the structured singular values of rank one matrices in $M_n(\C)$.

\begin{thm} \label{thm:sec2-01}
	For a linear subspace $E$ of $M_n(\C)$ with $E_{\text{scalar}} \subseteq E$, the following are equivalent:
	\begin{enumerate}
		\item $\mu_E(A)=r(A)$ for all $A \in M_n(\C)$;
		\item $E=E_{\text{scalar}}$.
	\end{enumerate} 
\end{thm}

\begin{proof}
	The part $(2) \implies (1)$ is a well-known fact in control theory. We prove $(1) \implies (2)$. Assume that $\mu_E(A)=r(A)$ for all $A \in M_n(\C)$. Suppose $X=(x_{ij})_{i, j=1}^n \in E$ has a non-zero off diagonal entry $x_{qp }$ for $1 \leq p, q \leq n$ with $p \ne q$. Consider $A=e_pe_q^t$, which is a matrix with $1$ at $pq$-th position and zero elsewhere. Since $A$ is a nilpotent matrix, $r(A)=0$. Let $\displaystyle Y=\frac{1}{x_{qp}}X, u=e_p$ and $v^t=e_q^tY$. Clearly, $Y \in E$ with $\|Y\|>0$ and $e_q^tYe_p$ is the $qp$-th entry of $Y$, i.e., $e_q^tYe_p=1$. Then 
	\[
	\det(I-AY)=\det(I-e_pe_q^tY)=\det(I-uv^t)=1-v^tu=1-e_q^tYe_p=0
	\]
	and so, there exists $Y \in E$ with $\det(I-AY)=0$. Consequently, we have
	\[
	\frac{1}{\mu_E(A)}=\inf\{\|Z\|: Z \in E, \det(I-AZ)=0\} \leq \|Y\| \quad \text{and so,} \quad \mu_E(A) \geq \|Y\|^{-1}>0,
	\]
	a contradiction as $\mu_E(A)=r(A)$ and $r(A)=0$. Thus, no matrix in $E$ has a non-zero off-diagonal entry and so, $E_{\text{scalar}} \subseteq E \subseteq E_{\text{diag}}$. Let if possible, $E \ne E_{\text{scalar}}$. Since $E_{\text{scalar}}=\text{span}\{\text{diag}(1, \dotsc, 1)\}$, one can find $v=(v_1, \dotsc, v_n)^t \in \C^n$ such that $v$ and $(1, \dotsc, 1)^t$ are linearly independent vectors, and $S=\text{span}\left\{\text{diag}(1, \dotsc, 1), \text{diag}(v_1, \dotsc, v_n) \right\} \subseteq E$. Moreover, there exists $i \in \{1, \dotsc, n-1\}$ such that $v_i \ne v_{i+1}$. Choose
	\[
	\alpha=v_i\left(1-\frac{1}{v_{i+1}-v_i}\right) \quad \text{and} \quad \beta=\frac{1}{v_{i+1}-v_i}.
	\]
	It is not difficult to see that $\alpha+\beta v_i=v_i$ and $\alpha+\beta v_{i+1}=1+v_i$. Consider the matrix given by
	\[
	X_{\alpha\beta}=\text{diag}(\alpha+\beta v_1, \dotsc, \alpha+\beta v_n)=\text{diag}(\alpha+\beta v_1, \ \alpha+\beta v_{i-1}, \ v_i, \ 1+v_i, \ \alpha+\beta v_{i+2}, \dotsc, \alpha+\beta v_n).
	\]
	Evidently, $X_{\alpha\beta}=\alpha \ \text{diag}(1, \dotsc, 1)+\beta \ \text{diag}(v_1, \dotsc, v_n)$ and so, $X_{\alpha \beta} \in S$. Moreover, $X_{\alpha \beta} \ne 0$. Consider the block matrix given by 
	\[
	A=\begin{pmatrix}
		\Theta_{(i-1 )\times (i-1)} & \Theta_{(i-1) \times 2} & \Theta_{(i-1) \times (n-1-i)} \\
		\Theta_{2 \times (i-1)}  & A_0 & \Theta_{2 \times (n-1-i)} \\
		\Theta_{(n-1-i) \times (i-1)} & \Theta_{(n-1-i) \times 2} & \Theta_{(n-1-i)\times (n-1-i)}
	\end{pmatrix}, \quad \text{where} \quad A_0=\begin{pmatrix}
		-1 & -1\\
		1 & 1
	\end{pmatrix}.
	\]
	Clearly, $\sigma(A) \subseteq \sigma(A_0) \cup \{0\}=\{0\}$ since $\sigma(A_0)=\{0\}$, and thus $r(A)=0$. Let $V=\text{diag}(v_i, 1+v_i)$. A straightforward computation gives that
	\[
	I-A_0V=\begin{pmatrix}
		1+v_i & 1+v_i\\
		-v_i & -v_i
	\end{pmatrix}
	\quad \text{and} \quad I-AX_{\alpha\beta}=\begin{pmatrix}
		I_{i-1 } & \Theta_{(i-1) \times 2} & \Theta_{(i-1) \times (n-1-i)} \\
		\Theta_{2 \times (i-1)}  & I-A_0V & \Theta_{2 \times (n-1-i)} \\
		\Theta_{(n-1-i) \times (i-1)} & \Theta_{(n-1-i) \times 2} & I_{n-1-i}
	\end{pmatrix}.
	\]
	Hence, $\det(I-AX_{\alpha\beta})=\det(I-A_0V)=0$. Putting everything together, we have $S \subseteq E$,
	\begin{align*}
		\frac{1}{\mu_S(A)}=\inf\{\|X\|: X \in S, \det(I-AX)=0\} \leq \|X_{\alpha\beta}\|	\quad \text{and so,} \quad \mu_E(A) \geq \mu_S(A) \geq \|X_{\alpha \beta}\|^{-1}>0,
	\end{align*}
	which is a contradiction as $\mu_E(A)=r(A)$ and $r(A)=0$. Therefore, $E=E_{\text{scalar}}$. 
\end{proof}

The following lemma can be considered an elementary result in linear algebra. We include a short proof here for the sake of completeness.

\begin{lem}\label{lem_ED1}
If $E$ is a proper subspace of $E_{\text{diag}}$, then there exists $W=\text{diag}(w_1, \dotsc, w_n)$ such that $w_1, \dotsc, w_n \in \T$ and $W \notin E$.
\end{lem}

\begin{proof}
	We have a natural identification of $E_{\text{diag}}$ with $\C^n$ via  $\text{diag}(z_1, \dotsc, z_n) \mapsto (z_1, \dotsc, z_n)^t$. So, there is an inner product structure on $E_{\text{diag}}$ given by
	$
	\langle Z, W \rangle=\overset{n}{\underset{j=1}{\sum}}z_j\overline{w}_j  
	$ 
	for $Z=\text{diag}(z_1, \dotsc, z_n)$ and $W=\text{diag}(w_1, \dotsc, w_n)$ in $E_{\text{diag}}$. Since $E$ is a proper subspace of $E_{\text{diag}}$, there is a non-zero $P=\text{diag}(p_1, \dotsc, p_n) \in E_{\text{diag}}$ such that $\langle Z, P \rangle=0$ for all $Z \in E$. Consider the linear functional given by $\xi: E_{\text{diag}} \to \C, \ \xi(Z)=\langle Z, P\rangle$ for all $Z \in E_{\text{diag}}$. Evidently, $\xi(E)=\{0\}$. Suppose $p_j=|p_j|e^{i\theta_j}$ for $1 \leq j \leq n$. Define 
	$
	w_j=\begin{cases}
		e^{i\theta_j}, & p_j \ne 0\\
		1, & p_j = 0\\
	\end{cases} \ $ for $1 \leq  j \leq n$. For $W=\text{diag}(w_1, \dotsc, w_n)$, we have
	\[
	\xi(W)=\langle W, P \rangle =\overset{n}{\underset{j=1}{\sum}}w_j\overline{p}_j= \underset{j: p_j \ne 0}{\sum}e^{i\theta_j}|p_j|e^{-i\theta_j}=\overset{n}{\underset{j=1}{\sum}}|p_j|>0
	\]
	and so, $W \notin E$. The proof is now complete.
\end{proof}

The next lemma appears in a more general form in Section 3 of \cite{Talha}. We give an alternative proof to the particular case considered here.

\begin{lem}\label{lem_ED2}
	For $u=(u_1, \dotsc, u_n)^t, v=(v_1, \dotsc, v_n)^t \in \C^n$ with $\overset{n}{\underset{j=1}{\sum}}|u_jv_j| \ne 0$,
	$
	\mu_{\text{diag}}(uv^*)=\overset{n}{\underset{j=1}{\sum}}|u_jv_j|.
	$  
\end{lem}

\begin{proof}
	Let $A=uv^*$ and $X=\text{diag}(x_1, \dotsc, x_n)$. Define $x=(x_1, \dotsc, x_n)^t$ and $c=(c_1, \dotsc, c_n)^t$ with $c_j=\overline{u}_jv_j$. Clearly, $\|c\|_1=\overset{n}{\underset{j=1}{\sum}}|u_jv_j|>0$. Since $\det(I-xy^t)=1-y^tx$ for vectors $x, y \in \C^n$, we have 
	\begin{align}\label{eqn_001}
		\det(I-AX)=\det(I-uv^*X)=1-v^*Xu=1-\overset{n}{\underset{j=1}{\sum}}x_ju_j\overline{v}_j=1- \langle x, c \rangle.
	\end{align}
	If $\det(I-AX)=0$, then $
	1=|\langle x, c\rangle|=\left|\overset{n}{\underset{j=1}{\sum}}x_ju_j\overline{v}_j \right| \leq \overset{n}{\underset{j=1}{\sum}}|x_j| \ |u_j\overline{v}_j|  \leq \|x\|_\infty \|c\|_1$ and so, $1 \slash \|c\|_1 \leq \|x\|_\infty$.	Thus, $1\slash \|c\|_1 \leq \inf\left\{\|X\|: X \in E_{\text{diag}}, \det(I-AX)=0 \right\}$. Let $c_j=|c_j|e^{i\theta_j}$ for $1 \leq j \leq n$. Choose $w=(w_1, \dotsc, w_n)$ given by
	$w_j=\begin{cases}
		\displaystyle \frac{e^{i\theta_j}}{\|c\|_1}, & c_j \ne 0\\
		0, & c_j = 0\\
	\end{cases}$ for $1 \leq  j \leq n$. Note that 
	$
	\|w\|_{\infty}=1\slash \|c\|_1$ and $\langle w, c \rangle=\underset{c_j \ne 0}{\sum}w_j\overline{c}_j=\frac{1}{\|c\|_1}\underset{c_j \ne 0}{\sum}|c_j|=1$. For $W=\text{diag}(w_1, \dotsc, w_n) \in E_{\text{diag}}$, we have by \eqref{eqn_001} that $\det(I-AW)=0$ and 
	$
	1\slash \|c\|_1 \leq \inf\left\{\|X\|: X \in E_{\text{diag}}, \ \det(I-AX)=0 \right\} \leq \|W\|=\|w\|_\infty=1\slash \|c\|_1.
	$ 
	Thus, $\mu_{\text{diag}}(A)=\|c\|_1$, which completes the proof.
\end{proof}

We are now in a position to present our second rigidity result for the structured singular value $\mu_\diag$ corresponding to the subspace $E_\diag$, which is the space of all diagonal matrices in $M_n(\C)$.

\begin{thm} \label{thm:sec2-02}
	For a linear subspace $E$ of $M_n(\C)$, the following are equivalent:
	\begin{enumerate}
		\item $\mu_E(A)=\mu_{\text{diag}}(A)$ for all $A \in M_n(\C)$;
		\item $E=E_{\text{diag}}$.
	\end{enumerate} 
\end{thm}

\begin{proof}
	The part $(2) \implies (1)$ is trivial. Suppose $\mu_E(A)=\mu_{\text{diag}}(A)$ for all $A \in M_n(\C)$. Let if possible, there exists $X=(x_{ij})_{i, j=1}^n \in E$ having a non-zero off diagonal entry $x_{qp }$ for $1 \leq p, q \leq n$ with $p \ne q$. Consider $A=e_pe_q^t$, which is a matrix with $1$ at $pq$-th position and zero elsewhere. Let $\displaystyle Y=\frac{1}{x_{qp}}X, u=e_p$ and $v^t=e_q^tY$. Clearly, $Y \in E$ with $\|Y\|>0$ and $e_q^tYe_p$ is the $qp$-th entry of $Y$, i.e., $e_q^tYe_p=1$. Then 
	\begin{align}\label{eqn_002}
		\det(I-AY)=\det(I-e_pe_q^tY)=\det(I-uv^t)=1-v^tu=1-e_q^tYe_p=0
	\end{align}
	and so, there exists $Y \in E$ with $\det(I-AY)=0$. Consequently, we have
	\[
	\frac{1}{\mu_E(A)}=\inf\{\|Z\|: Z \in E, \ \det(I-AZ)=0\} \leq \|Y\| \quad \text{and so,} \quad \mu_E(A) \geq \|Y\|^{-1}>0.
	\]
	For any $D \in E_{\text{diag}}$, we have that $e_q^tDe_p=0$. Now, applying the same arguments as in \eqref{eqn_002} gives
	$
	\det(I-AD)=\det(I-e_pe_q^tD)=1-e_q^tDe_p=1
	$
	and so, $\mu_{\text{diag}}(A)=0$. This is a contradiction since $\mu_E(A)=\mu_{\text{diag}}(A)$ and $\mu_E(A)>0$. Therefore, $E \subseteq E_{\text{diag}}$.

	\medskip 
	
	Let if possible, $E$ be a proper subspace of $E_{\text{diag}}$. It follows from Lemma \ref{lem_ED1} that there exists $W=\text{diag}(w_1, \dotsc, w_n)$ such that $w_1, \dotsc, w_n \in \T$ and $W \notin E$. Choose $u=(1, \dotsc, 1)^t$ and $w=(w_1, \dotsc, w_n)^t$. Note that $\overset{n}{\underset{j=1}{\sum}}|u_jw_j|=n$. For $B=uw^*$, we have by Lemma \ref{lem_ED2} that $\mu_{\text{diag}}(B)=\overset{n}{\underset{j=1}{\sum}}|u_jw_j|=n$ and so, $\mu_E(B)=n$. 
	Hence, $\inf\{\|X\|: X \in E, \det(I-BX)=0\}=1\slash n$. By compactness arguments, there exists $X \in E$ with $\det(I-BX)=0$ such that $\|X\|=1\slash n$. Let $X=\text{diag}(x_1, \dotsc, x_n)$ and $x=(x_1, \dotsc, x_n)^t$. Then $\|X\|=\|x\|_\infty=1 \slash n=1\slash \|w\|_1$. Applying similar arguments as in \eqref{eqn_001}, we have 
	\[
	\det(I-BX)=1-\overset{n}{\underset{j=1}{\sum}}x_j\overline{w}_j=1-\langle x, w\rangle \quad \text{and thus,} \quad \langle x, w\rangle=1.
	\] 
	So, $1=\left|\langle x, w\rangle \right|=\|x\|_\infty \|w\|_1$. Let if possible, there exist $k \in \{1, \dotsc, n\}$ such that $|x_k| <\|x\|_\infty$. Then 
	\[
	\|x\|_\infty \|w\|_1=\left|\langle x, w\rangle \right| \leq \overset{n}{\underset{j=1}{\sum}}|x_j| \ |\overline{w}_j|=\overset{n}{\underset{j=1}{\sum}}|x_j| <n\|x\|_\infty=\|x\|_\infty \|w\|_1,
	\]
	which is a contradiction. So, $x=(x_1, \dotsc, x_n)^t$ satisfies $|x_j|=\|x\|_\infty=1\slash n$ for $1 \leq j \leq n$. Note that 
	\[
	\|x\|_2=\left(\overset{n}{\underset{j=1}{\sum}}|x_j|^2\right)^{1\slash 2}=\frac{1}{\sqrt{n}} \quad \text{and} \quad \|w\|_2=\left(\overset{n}{\underset{j=1}{\sum}}|w_j|^2\right)^{1\slash 2}=\sqrt{n}.
	\]
	Thus, $|\langle x, w \rangle| =\|x\|_2\|w\|_2=1$. By Cauchy-Schwarz inequality, $x$ and $w$ are linearly dependent vectors, i.e., $w=ax$ for some $a \in \C \setminus \{0\}$. Consequently, $W=aX$ and thus, $W \in E$ as $X \in E$. This gives a contradiction to the fact that $W \notin E$. Hence, $E=E_{\text{diag}}$ and the proof is now complete.
\end{proof}

We now present our rigidity results for the structured singular values $\mu_\penta$ and $\mu_\hexa$  associated with the linear subspaces of $M_2(\C)$ given by
\[
E_{\text{penta}}=\left\{\begin{pmatrix}
	z & w\\
	0 & z
\end{pmatrix}: z, w  \in \C \right\} \quad \text{and} \quad E_{\text{hexa}}=\left\{\begin{pmatrix}
z_1 & w\\
0 & z_2
\end{pmatrix}: z_1, z_2, w  \in \C \right\},
\] 
respectively. The proof to these rigidity theorems share the similar arguments with each other.

\begin{thm}
	For a linear subspace $E$ of $M_2(\C)$ with $E_{\text{scalar}} \subseteq E$, the following are equivalent:
	\begin{enumerate}
		\item $\mu_E(A)=\mu_{\text{penta}}(A)$ for all $A \in M_2(\C)$;
		\item $E=E_{\text{penta}}$.
	\end{enumerate} 
\end{thm}

\begin{proof}
	The part $(2) \implies (1)$ is trivial. Now, suppose $\mu_E(A)=\mu_{\text{penta}}(A)$ for all $A \in M_2(\C)$. Suppose $X=(x_{ij})_{i, j=1}^2 \in E$ has the off diagonal entry $x_{21} \ne 0$. Define $A=e_1e_2^t$ and $Y=x_{21}^{-1}X$.  Clearly $Y \in E, \|Y\|>0$ and $e_2^tYe_1=1$. Thus, $\det(I-AY)=\det(1-e_1e_2^tY)=1-e_2^tYe_1=0$. We have
	\begin{align}\label{eqn_103}
		\frac{1}{\mu_E(A)}=\inf\{\|Z\|: Z \in E, \det(I-AZ)=0\} \leq \|Y\| \quad \text{and so,} \quad \mu_E(A)\geq \|Y\|^{-1}>0.
	\end{align}
	Let $Z=\begin{pmatrix}
		z& w\\ 0 & z
	\end{pmatrix} \in E_{\text{penta}}$. Then $\det(I-AZ)=\det(I-e_1e_2^tZ)=1-e_2^tZe_1=1$ and so, $\mu_{\text{penta}}(A)=0$. Therefore, $0=\mu_{\text{penta}}(A)=\mu_E(A)$, a contradiction to \eqref{eqn_103}. Putting everything together, we have 
	\[
	E_{\text{scalar}} \subseteq E \subseteq \left\{\begin{pmatrix}
		z_1 & w\\
		0 & z_2
	\end{pmatrix}: z_1, z_2, w \in \C \right\}=E_{\text{hexa}}.
	\]
	It follows from Proposition \ref{prop_mu} that $E_{\scal} \subsetneq E \subsetneq E_\hexa$. Then $\dim(E)=2$ and so, we can write 
	\[
	E=\text{span}\left\{I=\begin{pmatrix}
		1 & 0\\
		0 & 1
	\end{pmatrix},  \ M=\begin{pmatrix}
		a & b \\
		0 & c
	\end{pmatrix}\right\},
	\]	
	where $M$ and $I$ are linearly independent in $M_2(\C)$. We now discuss two cases depending on $a, b$ and $c$.
	
	\medskip 
	
	\noindent \textbf{Case 1.} Let $a \ne c$ and $b=0$. In this case, we have $E=E_{\diag}$ in $M_2(\C)$ and so, $\mu_{\penta}=\mu_E=\mu_{\diag}$, which contradicts Proposition \ref{prop_mu}.
	
	\medskip 
	
	\noindent \textbf{Case 2.} Let $a \ne c$ and $b \ne 0$. Let $B=\begin{pmatrix}
		0 & 0 \\
		1 & 0
	\end{pmatrix}$. It follows from the proof of Proposition \ref{prop_mu} that $\mu_\penta(B)=1$ and so, $\mu_E(B)=1$. Therefore, $\inf\{\|X\|: X \in E, \det(I-BX)=0\}=1$. By compactness arguments, there exists $X \in E$ with $\det(I-BX)=0$ such that $\|X\|=1$. One can choose $\alpha, \beta \in \C$ such that 
	\[
	X=\alpha I_2 +\beta M=\begin{pmatrix}
		\alpha + a \beta & b \beta \\
		0 & \alpha + c \beta 
	\end{pmatrix}. 
	\]
	A simple calculation shows that
	\[
	I-BX=\begin{pmatrix}
		1 & 0\\
		-\alpha-a\beta & 1-b\beta 
	\end{pmatrix} \quad \text{and so,} \quad \det(I-BX)=1-b \beta =0.
	\]
	Therefore, $\beta=1\slash b$ and $X=\begin{pmatrix}
		\alpha+a\beta & 1\\
		0 & \alpha+c\beta
	\end{pmatrix}$. We have by $\|X\|=1$ that $\alpha+c\beta=0=\alpha+a\beta$ and so, $(a-c)\beta=0$, which is a contradiction.
	
	\medskip 
	
	Therefore, the only remaining possibility is that $a=c$. Clearly $b \ne 0$, otherwise $M=aI_2$ leading to a contradiction. Consequently, $a=c$ and $b\ne 0$. Then
	\[
	E=\left\{\begin{pmatrix}
		z& w \\
		0 & z
	\end{pmatrix}: z, w \in \C\right\},
	\]
	which is the subspace $E_\penta$. The proof is now complete.
\end{proof}

\begin{thm}
	For a linear subspace $E$ of $M_2(\C)$ with $E_{\text{scalar}} \subseteq E$, the following are equivalent:
	\begin{enumerate}
		\item $\mu_E(A)=\mu_{\text{hexa}}(A)$ for all $A \in M_2(\C)$;
		\item $E=E_{\text{hexa}}$.
	\end{enumerate} 
\end{thm}

\begin{proof}
	The part $(2) \implies (1)$ is trivial. Assume that $\mu_E(A)=\mu_{\text{hexa}}(A)$ for all $A \in M_2(\C)$. Suppose $X=(x_{ij})_{i, j=1}^2 \in E$ has the off diagonal entry $x_{21} \ne 0$. Define $A=e_1e_2^t$ and $Y=x_{21}^{-1}X$.  Clearly $Y \in E, \|Y\|>0$ and $e_2^tYe_1=1$. Thus, $\det(I-AY)=\det(1-e_1e_2^tY)=1-e_2^tYe_1=0$. We have
	\begin{align}\label{eqn_104}
		\frac{1}{\mu_E(A)}=\inf\{\|Z\|: Z \in E, \det(I-AZ)=0\} \leq \|Y\| \quad \text{and so,} \quad \mu_E(A)\geq \|Y\|^{-1}>0.
	\end{align}
	Let $Z=\begin{pmatrix}
		z_1& w\\ 0 & z_2
	\end{pmatrix} \in E_{\text{hexa}}$. Then $\det(I-AZ)=\det(I-e_1e_2^tZ)=1-e_2^tZe_1=1$ and so, $\mu_{\text{hexa}}(A)=0$. Therefore, $0=\mu_{\text{hexa}}(A)=\mu_E(A)$, a contradiction to \eqref{eqn_104}. Consequently, we have 
	\[
	E_{\text{scalar}} \subseteq E \subseteq \left\{\begin{pmatrix}
		z_1 & w\\
		0 & z_2
	\end{pmatrix}: z_1, z_2, w \in \C \right\}=E_{\text{hexa}}.
	\]
	It follows from Proposition \ref{prop_mu} that $E_{\scal} \subsetneq E$. Let if possible, $\dim(E)=2$. Then we can write 
	\[
	E=\text{span}\left\{I=\begin{pmatrix}
		1 & 0\\
		0 & 1
	\end{pmatrix},  \ M=\begin{pmatrix}
		a & b \\
		0 & c
	\end{pmatrix}\right\},
	\]	
	where $M$ and $I$ are linearly independent in $M_2(\C)$. We now discuss three cases depending on $a, b$ and $c$.
	
	\medskip 
	
	\noindent \textbf{Case 1.} Let $a \ne c$ and $b=0$. In this case, we have $E=E_{\diag}$ in $M_2(\C)$ and so, $\mu_{\hexa}=\mu_E=\mu_{\diag}$, which contradicts Proposition \ref{prop_mu}.
	
	\medskip 
	
	\noindent \textbf{Case 2.} Let $a \ne c$ and $b \ne 0$. Let $B=\begin{pmatrix}
		0 & 0 \\
		1 & 0
	\end{pmatrix}$. It follows from the proof of Proposition \ref{prop_mu} that $\mu_\hexa(B)=1$ and so, $\mu_E(B)=1$. Therefore, $\inf\{\|X\|: X \in E, \det(I-BX)=0\}=1$. By compactness arguments, there exists $X \in E$ with $\det(I-BX)=0$ such that $\|X\|=1$. One can choose $\alpha, \beta \in \C$ such that 
	\[
	X=\alpha I_2 +\beta M=\begin{pmatrix}
		\alpha + a \beta & b \beta \\
		0 & \alpha + c \beta 
	\end{pmatrix}. 
	\]
	A simple calculation shows that
	\[
	I-BX=\begin{pmatrix}
		1 & 0\\
		-\alpha-a\beta & 1-b\beta 
	\end{pmatrix} \quad \text{and so,} \quad \det(I-BX)=1-b \beta =0.
	\]
	Therefore, $\beta=1\slash b$ and $X=\begin{pmatrix}
		\alpha+a\beta & 1\\
		0 & \alpha+c\beta
	\end{pmatrix}$. We have by $\|X\|=1$ that $\alpha+c\beta=0=\alpha+a\beta$ and so, $(a-c)\beta=0$, which is a contradiction.
	
	\medskip 
	
	\noindent \textbf{Case 3.} Let $a=c$. If $b = 0$, then $M=aI$, which is a contradiction. Hence, $a=c$ and $b \ne 0$. Then
	\[
	E=\left\{\begin{pmatrix}
		z& w \\
		0 & z
	\end{pmatrix}: z, w \in \C\right\}=E_\penta.
	\]
	In this case, $\mu_\penta=\mu_\hexa$, which is a contradiction to Proposition \ref{prop_mu}. 
	
	\medskip 
	
	\noindent The above three cases show that $\dim(E)>2$ and so, $E=E_\hexa$. The proof is now complete.
\end{proof}

\section{The operator norm and structured singular value}\label{sec_op}

\noindent In this section, we consider the problem of identifying the linear subspaces $E \subseteq M_n(\C)$ for which the structured singular value $\mu_E$ coincides with the operator norm. The space $M_n(\C)$ is a well-known example, whereas in contrast to earlier rigidity phenomena, we show that proper subspaces of $M_n(\C)$ with this property also exist. We begin with the following characterization of all subspaces for which the equality $\mu_E=\|.\|$ holds.

\begin{thm}\label{thm_norm}
	For a linear subspace $E$ of $M_n(\C)$, the following are equivalent:
	\begin{enumerate}
		\item $\mu_E(A)=\|A\|$ for all $A \in M_n(\C)$;
		\item for every pair of unit vectors $u, v \in \C^n$, there exists $X \in E$ with $\|X\|=1$ such that $Xu=v$.
	\end{enumerate} 
\end{thm}

\begin{proof}
	$(1) \implies (2)$. Let $\mu_E(A)=\|A\|$ for all $A \in M_n(\C)$. Take unit vectors $u, v \in \C^n$ and define $M=uv^*$. Clearly, $\|M\|=1$ and $Mv=u$. By hypothesis, $\mu_E(M)=1$ and so, $\inf\{\|X\|: X \in E, \det(I-AX)=0\}=1$. By compactness arguments, there exists $X \in E$ with $\det(I-AX)=0$ such that $\|X\|=1$. For any $z \in \C^n$, we have
	\[
	MXz=u(v^*Xz)=\alpha_0 u, \quad \text{where $\alpha_0=v^*Xz \in \C$}.
	\]
	Thus, the range of $MX$ is contained in the subspace $\text{span}\{u\}$. Since $\det(I-MX)=0$, there exists $y \in \C^n$ with $\|y\|=1$ such that $y=MXy$. Hence, $y$ lies in the range of $MX$ and thus, one can choose $\alpha \in \C \setminus \{0\}$ such that $y=\alpha u$. Note that $MX(\alpha u)=MXy=y=\alpha u$ and so, $\alpha u (v^*Xu)=\alpha MXu=\alpha u$. Since $\alpha \ne 0$ and $u^*u=1$, we have that $v^*Xu=1$, i.e., $\langle Xu, v \rangle=1$. By Cauchy-Schwarz inequality, we have
	\[
	1=|\langle Xu, v \rangle| \leq \|Xu\| \ \|v\| \leq \|X\| \ \|u\| \|v\|=1
	\]
	and thus, $v$ and $Xu$ are linearly dependent vectors in $\C^n$. One can now choose $\beta \in \C$ so that $Xu=\beta v$. Using the facts that $\langle Xu, v \rangle=1$ and $\|v\|=1$, we have $\beta=1$. Therefore, $Xu=v$.
	
	\medskip 
	
	\noindent $(2) \implies (1)$. The proof is essentially based on the discussion after Lemma 3.7 in \cite{Packard}. For any $A \in M_n(\C)$, it is well-known that $\mu_E(A) \leq \|A\|$. Let $A \in M_n(\C)$. If $A=0$, then $\mu_E(A)=\|A\|=0$. Assume that $A \ne 0$. Choose unit vectors $u, v$ in $\C^2$ such that $\|Av\|=\|A\|$ and $Av=\|A\|u$. By given hypothesis, there exists $S \in E$ with $\|S\|=1$ such that $Su=v$. Define $X=\|A\|^{-1}S$. Note that $X \in E$ and $\|X\|=1 \slash \|A\|$. Then
	$\displaystyle AXu=\|A\|^{-1}ASu=\|A\|^{-1}Av=u$ and $\det(I-AX)=0$. Consequently, $\displaystyle \mu_E(A)^{-1}=\inf \left\{\|Y\|: Y \in E, \ \det(I-AY)=0\right\} \leq \|X\|=\|A\|^{-1}$ and so, $\|A\| \leq \mu_E(A)$. 
\end{proof}

Our next result provides a lower bound on the dimension of $E$ for which $\mu_E=\|.\|$.

\begin{cor}\label{cor_dimn}
	For a linear subspace $E$ of $M_n(\C)$ with $\mu_E=\|.\|$, we have $\dim(E)\geq n$.
\end{cor}

\begin{proof}
	Let $u$ be a unit vector in $\C^n$. Consider the linear map $\varphi: E \to \C^n$ given by $\varphi(A)=Au$. Let $v$ be a non-zero vector in $\C^n$. Theorem \ref{thm_norm} guarantees the existence of $X \in E$ such that $Xu=v\slash \|v\|$ and so, $\varphi(\|v\|X)=\|v\|Xu=v$. Hence, $\varphi$ is surjective and the desired conclusion holds. 
\end{proof}

We now present concrete examples of proper linear subspaces $E \subseteq M_2(\C)$ such that $\mu_E=\|.\|$. To this end, Theorem \ref{thm_norm} reduces the task to constructing subspaces $E \subseteq M_2(\C)$ whose unit norm matrices act transitively on the unit sphere of $\C^2$. In this direction, our first example shows that $\mu_E=\|.\|$ when $E=E_{\text{symm}}$ in $M_2(\C)$.

\begin{eg}\label{lem_501}
	Let $u, v$ be unit vectors in $\C^2$. Choose a unitary $W \in M_2(\C)$ such that $We_1=u$. Define 
$
w=W^tv=\begin{pmatrix}\alpha \\ \beta \end{pmatrix}$ for some $\alpha, \beta \in \C$. Clearly, $\|w\|^2=|\alpha|^2+|\beta|^2=1$. Choose $\gamma \in \T$ such that $\overline{\beta}=\gamma\beta$ and define $B=\begin{pmatrix}
	\alpha & \beta\\
	\beta & -\overline{\gamma}\ \overline{\alpha}
\end{pmatrix}$. Note that $B$ is a symmetric unitary matrix with $Be_1=w=W^tv$. Hence, the symmetric unitary matrix $A=(W^*)^tBW^*$ satisfies $Au=(W^*)^tBW^*u=(W^*)^tBe_1=(W^*)^tW^tv=(WW^*)^tv=v$. We have by Theorem \ref{thm_norm} that $\mu_{E_{\text{symm}}}=\|.\|$. \qed
\end{eg} 

We now present a generalization of the above example producing a family of proper linear subspaces $E$ in $M_2(\C)$ for which $\mu_E=\|.\|$.

\begin{eg}\label{cor_Etheta}
	Let $\theta \in \R$ and let 
	\[
	E_\theta=\left\{\begin{pmatrix}
		z_1 & w \\
		e^{i\theta} w & z_2
	\end{pmatrix} : z_1, z_2, w \in \C \right\}.
	\]
For the unitary matrix given by $S=\begin{pmatrix}
		1 & 0 \\
		0 & e^{i\theta \slash 2}
	\end{pmatrix}$, we have that $E_{\text{symm}}=\{S^*A S : A \in E_\theta\}$. Let $u, v$ be unit vectors in $\C^2$. Define $u_1=S^*u$ and $v_1=S^*v$. Clearly, $u_1$ and $v_1$ are unit vectors in $\C^2$.	By Example \ref{lem_501}, there is a symmetric unitary matrix $B$ in $M_2(\C)$ such that $Bu_1=v_1$. For $A=SBS^*$, we have that $A$ is a unitary matrix in $ E_\theta$ and $Au=SBS^*u=SBu_1=Sv_1=v$. Consequently, it follows from Theorem \ref{thm_norm} that $\mu_{E_{\theta}}=\|.\|$. \qed 
\end{eg}

The following example gives another family of proper linear subspaces in $M_2(\C)$ for which the associated structured singular value coincides with the operator norm.

\begin{eg}\label{cor_Munitary}
	Consider the linear subspace of $M_2(\C)$ given by 
	\[
	E_M=\{A \in M_2(\C) : \text{tr}(M^*A)=0\},
	\]
	where $M \in M_2(\C)$ is a unitary matrix. We show that $\mu_{E_M}=\|.\|$ and by Theorem \ref{thm_norm}, it suffices to show that for any unit vectors $u, v \in \C^2$, there exists a unitary  matrix $A \in M_2(\C)$ such that $\text{tr}(M^*A)=0$ and $Au=v$. To this end, let $u, v$ be unit vectors in $\C^2$ and let $v_0=M^*v$. Choose a unitary $W \in M_2(\C)$ such that $We_1=u$. Define 
	$
	w=W^*v_0=\begin{pmatrix}\alpha \\ \beta \end{pmatrix}$ for some $\alpha, \beta \in \C$. Evidently, $\|w\|^2=|\alpha|^2+|\beta|^2=1$. Choose $\gamma \in \T$ such that $\alpha=-\gamma\overline{\alpha}$ and define $B=\begin{pmatrix}
		\alpha & -\gamma \overline{\beta}\\
		\beta & \gamma \overline{\alpha}
	\end{pmatrix}$. It is easy to see that $B$ is a unitary matrix with $\text{tr}(B)=0$ and $Be_1=w=W^*v_0$. Consequently, the unitary matrix $U=WBW^*$ has trace zero and 
	$Uu=WBW^*u=WBe_1=WW^*v_0=v_0=M^*v$. It is easy to see that $A=MU$ is a unitary matrix in $E_M$ with $Au=v$. \qed
\end{eg}

All the proper linear subspaces $E \subseteq M_2(\C)$ for which $\mu_E=\|.\|$ as mentioned in the above examples have dimension three. While Corollary \ref{cor_dimn} shows that any such linear subspace $E$ has dimension at least two, we shall prove in Corollary \ref{cor_303} that dimension of $E$ cannot be equal to $2$.  We conclude this section by providing an analog of  Example \ref{lem_501} in higher dimensions, and prove that $\mu_E=\|.\|$ when $E$ is a linear subspace of $M_n(\C)$ containing the subspace $E_{\text{symm}}$. 

\begin{thm}\label{thm_transitive}
	For unit vectors $u, v \in \C^n$, there exists a symmetric unitary matrix $U \in M_n(\C)$ such that $Uu=v$. 	In particular, if $E$ is a linear subspace of $M_n(\C)$ with $E_{\text{symm}} \subseteq E$, then $\mu_E=\|.\|$.
\end{thm}

\begin{proof}
	For a given a complex matrix $W$ of order $m \times n$, we denote by $\overline{W}=(W^*)^t$. Let $u, v$ be unit vectors in $\C^n$ and $\zeta=u^tv$.  Then $|\zeta|=|u^tv|=|\overline{u}^*v|=\left| \langle \overline{u}, v \rangle \right| \leq \|\overline{u}\| \|v\| \leq 1$. We can write $\zeta=|\zeta|e^{i\theta}$ for some $\theta \in \R$. Choose $y=(y_1, y_2, 0, \dotsc, 0)^t \in \C^n$ with 
	\[
	y_1=e^{i\theta\slash 2}\sqrt{\frac{1+|\zeta|}{2}} \quad \text{and} \quad y_2=ie^{i\theta\slash 2}\sqrt{\frac{1-|\zeta|}{2}}.
	\]
	It is clear that $\|y\|=\sqrt{|y_1|^2+|y_2|^2}=1$ and $y^ty=y_1^2+y_2^2=|\zeta| e^{i\theta}=\zeta$. For $u_1=u, v_1=y, u_2=\overline{v}, v_2=\overline{y}$, we have that $u_i^*u_j=v_i^*v_j$ for $1 \leq i, j \leq 2$. In other words, the Gram matrices $[u_i^*u_j]_{i, j=1}^2$ and $[v_i^*v_j]_{i, j=1}^2$ coincide. Consequently, one can find a unitary $W \in M_n(\C)$ such that $Wu_i=v_i$ for $ 1 \leq i, j \leq 2$. So, we have $Wu=y$ and $W\overline{v}=\overline{y}$. Also, $\overline{v}=W^*\overline{y}$ and thus, $v=W^ty$. Evidently, the symmetric unitary matrix $U=W^tW$ satisfies $Uu=W^tWu=W^ty=v$.  Since $E_{\text{symm}} \subseteq E \subseteq M_n(\C)$, we have that $\mu_{E_{\text{symm}}}(A) \leq \mu_E(A) \leq \|A\|$ for all $A \in M_n(\C)$. The rest of the conclusion follows from Theorem \ref{thm_norm}, which completes the proof.
\end{proof}

Evidently, Example \ref{lem_501} and Theorem \ref{thm_transitive} both capitalize Theorem \ref{thm_norm} establishing the transitivity of unit norm matrices on the unit sphere of $\C^2$. While Example \ref{lem_501} gives an explicit construction of such a unit norm matrix for a given pair of unit vectors in $\C^2$, the arguments presented in Theorem \ref{thm_transitive} for unit vectors in $\C^n$ give existence of such a unit norm matrix. 

\section{Structured singular value and the numerical radius} \label{sec-04-new}

\noindent In this section, we investigate if there is a linear subspace $E$ of $M_n(\C)$ such that the associated structured singular value $\mu_E$ coincides with the numerical radius. Recall that the numerical radius of a matrix $A \in M_n(\C)$ is defined by $w(A)=\sup\{|x^*Ax| : x \in \C^n, \|x\|=1\}$. It is well-known that the numerical radius defines a norm on $M_n(\C)$ that is equivalent to the operator norm, satisfying
\[
\max\{r(A), \|A\| \slash 2\} \leq w(A) \leq \|A\| \quad \text{for all } A \in M_n(\C).
\]
Here, we prove that even for $n=2$ there is no linear subspace $E \subseteq M_2(\C)$ such that $\mu_E(A)=w(A)$ for all $A \in M_2(\C)$. Surprisingly, more is true; we prove that there is no linear subspace $E \subseteq M_2(\C)$ for which $\mu_E$ is a convex combination of the spectral radius, numerical radius, and operator norm, except when $\mu_E$ coincides with either the spectral radius or the operator norm. We need a few elementary results to reach our goal.

\begin{lem}\label{lem_301}
	If $E$ is a linear subspace of $M_2(\C)$ with $\dim(E) \geq 2$, then $E$ has a rank one matrix.
\end{lem}

\begin{proof}
	Since $\dim(E) \geq 2$, there exist two linearly independent matrices $A$ and $B$ in $E$. By fundamental theorem of algebra, one can choose $\alpha \in \C$ such that the polynomial $p(z)=\det(Az+B)$ has $\alpha$ as a root. Then $\alpha A+B$ is non-invertible and so, it has rank at most $1$. Since $A$ and $B$ are linearly independent, $\alpha A+B$ cannot have rank $0$ and thus, the rank of $\alpha A+B$ is one.
\end{proof}

Consider the space $M_n(\C)$ equipped with the inner product structure given by
\[
\langle A, B \rangle_{\text{HS}} =tr(B^*A), 
\]
for all $A, B \in M_n(\C)$. The inner product $\langle .,.\rangle_{\text{HS}}$ is commonly referred to as the \textit{Hilbert-Schmidt inner product}. For a linear subspace $E$ of $M_n(\C)$, we denote by
\[
E_\hp=\{X \in M_n(\C): \langle X, Y\rangle_{\text{HS}}=0 \ \text{for all $Y \in E$}\}=\{X \in M_n(\C): tr(Y^*X)=0 \ \text{for all $Y \in E$}\}.
\] 
The Hilbert-Schmidt inner product on $M_n(\C)$ plays a crucial role throughout the section. 

\begin{lem}\label{lem_302}
	Let $E$ be a linear subspace of $M_2(\C)$. If $E_\hp$ has a rank one matrix, then there exists a rank one matrix $A \in M_2(\C)$ such that $\mu_E(A)=0$.	
\end{lem}

\begin{proof}
	Suppose $Y \in E_\hp$ has rank one. Then $Y=vu^*$ for some non-zero vectors $u, v \in \C^2$. For every $X \in E$, we have $\langle X, Y\rangle_{\text{HS}}=0$. Consequently, 
	$
	\det(I-uv^*X)=1-v^*Xu=1-tr(uv^*X)=1-tr(Y^*X)=1-\langle X, Y\rangle_{\text{HS}}=1.
	$
	Hence, no $X \in E$ satisfies $\det(I-uv^*X)=0$ and so $\mu_E(uv^*)=0$. Since $uv^*$ is rank one, the desired conclusion follows.
\end{proof}

\begin{cor}\label{cor_303}
	If $E$ is a linear subspace of $M_2(\C)$ with $\dim(E) \leq 2$, then $\mu_E$ cannot define a norm on $M_2(\C)$.
\end{cor}

\begin{proof}
	It is evident that $\dim(E_\hp)= 4-\dim(E) \geq 2$. By Lemma \ref{lem_301}, $E_\hp$ has a rank one matrix. It follows from Lemma \ref{lem_302} that there is a rank one matrix $A \in M_2(\C)$ such that $\mu_E(A)=0$. Since $A$ has rank one and $\mu_E(A)=0$, the function $\mu_E$ cannot define a norm on $M_2(\C)$.
\end{proof}

Before going to the main result, we establish that no convex combination of the numerical radius and operator norm can arise as a structured singular value except for the operator norm. In this direction, we begin with following result.

\begin{prop}\label{prop_304}
	Let $E$ be a linear subspace of $M_2(\C)$ with $E_\diag \subseteq E$ and let $t \in (0, 1]$. Then there exists a matrix $A$ $($depending on $t)$ such that  $\mu_E(A) \ne tw(A)+(1-t)\|A\|$. 
\end{prop}

\begin{proof}
	Evidently, the function $f_t: M_2(\C) \to [0, \infty)$ given by 
	$f_t(A)=tw(A)+(1-t)\|A\|$ defines a norm on $M_2(\C)$. Let if possible, $\mu_E=f_t$. Since $E \supseteq E_\diag$, we have $\dim(E) \geq 2$. By Corollary \ref{cor_303}, $\dim(E) \ne 2$. If $\dim(E)=4$, then $\mu_E=\|.\|$ and so, $f_t(A)=tw(A)+(1-t)\|A\|=\|A\|$ for all $A \in M_2(\C)$. Therefore, $w(A)-\|A\|=0$ for all $A \in M_2(\C)$, which is not true in general. Hence, $\dim(E)=3$ and we have
	\[
	E=\text{span}\left\{\begin{pmatrix}
		1 & 0\\
		0 & 0
	\end{pmatrix}, \begin{pmatrix}
		0 & 0\\
		0 & 1
	\end{pmatrix}, \begin{pmatrix}
		0 & \alpha\\
		\beta & 0
	\end{pmatrix}\right\}
	\]
	for some $\alpha, \beta \in \C$. Since $w(E_{12})=w(E_{21})=1\slash 2$, we have that $f_t(E_{12})=f_t(E_{21})=1-t\slash 2$.  We now discuss four cases depending on the values of $\alpha$ and $\beta$.
	
	\medskip 
	
	\noindent \textbf{Case 1.} Let $\alpha=0$. Then $\beta \neq 0$, otherwise $\dim(E)=2$. Evidently, $E = \mathrm{span}\{E_{11}, E_{22}, E_{21}\}$. For any $X \in E$, $\det(I - E_{21}X) = 1$. Thus, $\mu_E(E_{21})= 0$, contradicting $\mu_E=f_t$ since $f_t(E_{21}) = 1-t\slash 2$.
	
	\medskip 
	
	\noindent \textbf{Case 2.} Let $\beta=0$. Then $\alpha \neq 0$, otherwise $\dim(E)=2$. Clearly, $E = \mathrm{span}\{E_{11}, E_{22}, E_{12}\}$. For any $X \in E$, $\det(I - E_{12}X) = 1$. Therefore, $\mu_E(E_{12})= 0$, which is a contradiction as $f_t(E_{12})= 1-t\slash 2$.
	
	\medskip 
	
	\noindent \textbf{Case 3.} Let $\alpha, \beta \ne 0$ with $|\alpha| \leq |\beta|$. We can simply write 
	\[
	E = \text{span}\left\{E_{11}, E_{22}, \begin{pmatrix}
		0 & x \\
		1 & 0
	\end{pmatrix}\right\}=\left\{\begin{pmatrix}
		a & \gamma \ x \\
		\gamma & b
	\end{pmatrix} : a, b, \gamma \in \C \right\},
	\]
	where $x=\alpha \slash \beta$. For every $X=\begin{pmatrix}
		a & \gamma \ x \\
		\gamma & b
	\end{pmatrix} \in E$, we have that $\det(I-E_{12}X)=1-\gamma$. Then
	\[
	1 \leq \inf\{\|X\|: X \in E, \det(I-E_{12}X)=0\}=\inf\left\{\left\|\begin{pmatrix}
		a & x \\
		1 & b
	\end{pmatrix}\right\| : a, b \in \C \right\} \leq \left\|\begin{pmatrix}
		0 & x \\
		1 & 0
	\end{pmatrix}\right\| \leq 1,
	\]
	where the last inequality holds as $|x| \leq 1$. Thus, $\mu_E(E_{12})=1$ contradicting $\mu_E=f_t$ as $f_t(E_{12}) < 1$.
	
	\medskip 
	
	\noindent \textbf{Case 4.} Let $\alpha, \beta \ne 0$ with $|\beta| \leq |\alpha|$. We can simply write 
	\[
	E = \text{span}\left\{E_{11}, E_{22}, \begin{pmatrix}
		0 & 1 \\
		y & 0
	\end{pmatrix}\right\}=\left\{\begin{pmatrix}
		a & \gamma  \\
		\gamma \ y & b
	\end{pmatrix} : a, b, \gamma \in \C \right\},
	\]
	where $y=\beta\slash \alpha$. For every $X=\begin{pmatrix}
		a & \gamma  \\
		\gamma \ y & b
	\end{pmatrix} \in E$, we have that $\det(I-E_{21}X)=1-\gamma$. Then
	\[
	1 \leq \inf\{\|X\|: X \in E, \det(I-E_{21}X)=0\}=\inf\left\{\left\|\begin{pmatrix}
		a & 1 \\
		y & b
	\end{pmatrix}\right\| : a, b \in \C \right\} \leq \left\|\begin{pmatrix}
		0 & 1 \\
		y & 0
	\end{pmatrix}\right\| \leq 1,
	\]
	where the last inequality holds as $|y| \leq 1$. Thus, $\mu_E(E_{21})=1$ contradicting $\mu_E=f_t$ as $f_t(E_{21}) < 1$.
	
	\medskip 
	
	In either case, we obtain a contradiction and thus, $\mu_E \ne f_t$. The proof is now complete.
\end{proof}

The following result shows that no convex combination of the numerical radius and the operator norm can be realized as a structured singular value, except for the operator norm. 

\begin{thm}\label{thm_numrigidI}
	Let $t \in (0, 1]$. If $f_t: M_2(\C) \to [0, \infty)$ is the function 
	$
	f_t(A)=tw(A)+(1-t)\|A\|,
	$ 
	then no linear subspace $E$ of $M_2(\C)$ satisfies $\mu_E=f_t$.
\end{thm}

\begin{proof}
	Clearly, $f_t$ defines a norm on $M_2(\C)$. Suppose $E$ is a linear subspace of $M_2(\C)$. Let if possible, $\mu_E(A)= f_t(A)$ for all $A \in M_2(\C)$. By Corollary \ref{cor_303}, $\dim(E)>2$. If $\dim(E)=4$, then $\mu_E=\|\cdot\|$. Hence, $f_t(A)=tw(A)+(1-t)\|A\|=\|A\|$ for all $A \in M_2(\mathbb{C})$. This would imply $w(A)=\|A\|$ for all $A \in M_2(\C)$, which is a contradiction. Thus, the only possibility is $\dim(E)=3$. Let $U$ be a unitary in $M_2(\C)$. Note that 
	\begin{align}\label{eqn_muEU}
		\mu_E(U^*AU)=f_t(U^*AU)=f_t(A)=\mu_E(A)
	\end{align}
	for all $A \in M_2(\C)$. For any non-zero $A \in M_2(\C)$, we have $\mu_E(U^*AU)=\mu_E(A)=f_t(A)>0$. Then	
	\begin{align*}
		\mu_E(U^*AU)^{-1}
		&=\inf\{\|X\|: X \in E,  \ \det(I-U^*AUX)=0\}\\
		&=\inf\{\|UXU^*\| : UXU^* \in UEU^*, \  \det(I-AUXU^*)=0 \}\\
		&=\inf\{\|Y\|: Y \in E_U, \ \det(I-AY)=0\}\\
		&=\mu_{E_U}(A)^{-1},
	\end{align*} 
	where $E_U$ is the linear subspace given by
	\[
	E_U=UEU^*=\{UXU^*: X \in E\}=\{Y \in M_2(\C) : U^*YU \in E\}.
	\]
	Since $\dim(E)=3$, we have $\dim(E_U)=3$ and so, $\dim(E_\hp)=\dim((E_U)_\hp)=1$ for every unitary $U \in M_2(\C)$. So, we can write $E_\hp=\text{span}\{M\}$. If $M$ has rank one, then the desired conclusion follows from Lemma \ref{lem_302}. We assume that $M$ has rank $2$. Then $M$ is invertible and we choose a unitary $U$ such that 
	\[
	UMU^*=\begin{pmatrix}
		\lm_1 & m_{12}\\
		0 & \lm_2 
	\end{pmatrix} \quad \text{and so,} \quad U\left(\lm_1^{-1}M\right)U^*=\begin{pmatrix}
		1 & x \\
		0 & y
	\end{pmatrix},
	\]
	where $x=\lm_1^{-1}m_{12}$ and $y=\lm_1^{-1}\lm_2$. Set $M_1=\lm_1^{-1}M$ and $N=UM_1U^*$. Clearly, $M_1 \in E_\hp$ and we have $E=\{X \in M_2(\C) : tr(M_1^*X)=0\}$. Also, it follows that
	\[
	E_U=\{Y \in M_2(\C) : \text{tr}(M_1^*U^*YU) =0\}=\{Y \in M_2(\C) : \text{tr}(N^*Y) =0\}
	\]
	and thus, $(E_U)_\hp=\text{span}\{N\}$. Moreover, we have by \eqref{eqn_muEU} that 
	\begin{align}\label{eqn_muEUII}
		\mu_{E_U}(A)=f_t(A) \quad \text{for all $A \in M_2(\C)$.}
	\end{align} 
	In particular, $\mu_{E_U}(E_{ii})=f_t(E_{ii})=1$ for $i=1, 2$. Hence, $\inf\{\|X\|: X \in E_U, \det(I-E_{ii}X)=0\}=1$ for $i=1, 2$. By compactness argument, we have the existence of $P=(p_{ij})_{i, j=1}^2, Q=(q_{ij})_{i, j=1}^2 \in E_U$ such that $\|P\|=\|Q\|=1$ and $\det(I-E_{22}P)=\det(I-E_{11}Q)=0$. A routine computation shows that $\det(I-E_{22}P)=1-p_{22}$ and $\det(I-E_{11}Q)=1-q_{11}$. Hence, we have $p_{22}=1=q_{11}$. Since $\|P\|=\|Q\|=1$, we must have $p_{12}=p_{21}=q_{12}=q_{21}=0$. Consequently, we have $P=\diag(p_{11}, 1)$ and $Q=\diag(1, q_{22})$. If $P$ and $Q$ are linearly independent in $M_2(\C)$, then $\text{span}\{P, Q\}=E_\diag$. Thus, $E_\diag \subseteq E_U$ and by Proposition \ref{prop_304}, $\mu_{E_U} \ne f_t$, which is a contradiction to \eqref{eqn_muEUII}. Hence, $P$ and $Q$ are linearly dependent matrices. For $W=(w_{ij}) \in M_2(\C)$, we have
	\[
	N^*W=\begin{pmatrix}
		1 & 0\\
		\overline{x} & \overline{y} 
	\end{pmatrix}\begin{pmatrix}
		w_{11} & w_{12}\\
		w_{21} & w_{22}
	\end{pmatrix}=\begin{pmatrix}
		w_{11} & *\\
		* & \overline{x}w_{12}+\overline{y}w_{22} 
	\end{pmatrix} \quad \text{and so,} \quad \text{tr}(N^*W)=w_{11}+\overline{x}w_{12}+\overline{y}w_{22}.
	\] 
	Thus, $\text{tr}(N^*P)=0$ gives $p_{11}=-\overline{y}$. We now find a basis of $E_U$ containing $P$. Note that
	\begin{align*}
		E_U=\{W \in M_2(\C) : \text{tr}(N^*W) =0\}
		&=\{(w_{ij}) \in M_2(\C): w_{11}+\overline{x}w_{12}+\overline{y}w_{22}=0 \}\\
		&=\left\{\begin{pmatrix}
			-\overline{x}w_{12}-\overline{y}w_{22} & w_{12}\\
			w_{21} & w_{22}
		\end{pmatrix} : w_{12}, w_{21}, w_{22} \in \C\right\}\\
		&=\text{span}\left\{\begin{pmatrix}
			-\overline{x} & 1\\
			0 & 0
		\end{pmatrix}, \begin{pmatrix}
			-\overline{y} & 0\\
			0 & 1
		\end{pmatrix}, \begin{pmatrix}
			0 & 0\\
			1 & 0
		\end{pmatrix}\right\}.
	\end{align*}
	For any $W=\begin{pmatrix}
		-\overline{x}w_{12}-\overline{y}w_{22} & w_{12}\\
		w_{21} & w_{22}
	\end{pmatrix} \in E_U$, we have that $\det(I-E_{12}W)=1-w_{21}$.  
	Then
	\begin{align*}
	&	1 \leq \inf\left\{\left\|\begin{pmatrix}
			-\overline{x}w_{12}-\overline{y}w_{22} & w_{12}\\
			1 & w_{22}
		\end{pmatrix}\right\| : w_{12}, w_{22} \in \C \right\} \leq \left\|\begin{pmatrix}
			0 & 0 \\
			1 & 0
		\end{pmatrix}\right\| =1 \quad \text{and so,}
\\
&	\inf\{\|W\|: W \in E_U, \det(I-E_{12}W)=0\}
	=\inf\left\{\left\|\begin{pmatrix}
		-\overline{x}w_{12}-\overline{y}w_{22} & w_{12}\\
		1 & w_{22}
	\end{pmatrix}\right\| : w_{12}, w_{22} \in \C \right\}=1.
\end{align*}
So, $\mu_{E_U}(E_{12})=1$, which is not possible due to \eqref{eqn_muEUII} as $f_t(E_{12})=1-t\slash 2$. In either case, we obtain a contradiction. Hence, the desired conclusion follows.
\end{proof}

We now generalize the above theorem showing in fact that no convex combination of the spectral radius, numerical radius and operator norm gives a structured singular value except for the spectral radius and the operator norm. The arguments presented here are similar to that of Theorem \ref{thm_numrigidI}.

\begin{thm}
	Let $t=(t_1, t_2) \in [0, 1] \times [0, 1]$ with $t_1+t_2 \leq 1$ and $(t_1, t_2) \notin \{(0, 0), (1, 0)\}$. If $h_t: M_2(\C) \to [0, \infty)$ is the function given by 
	\[
	h_t(A)=t_1r(A)+t_2w(A)+(1-t_1-t_2)\|A\|, 
	\]
	then there does not exist a linear subspace $E$ of $M_2(\C)$ such that $\mu_E(A)=h_t(A)$ for all $A \in M_2(\C)$.
\end{thm}

\begin{proof}
	First note that for non-zero $A \in M_2(\C)$, we have $h_t(A)>0$, otherwise $t_1r(A)=t_2w(A)=(1-t_1-t_2)\|A\|=0$. In this case, $(t_1, t_2)=(1, 0)$, which is not possible. Let if possible, $h_t=\|.\|$, then $h_t(E_{12})=1-t_1-t_2\slash 2=\|E_{12}\|=1$. Thus, $(t_1, t_2)=(0, 0)$, which is again a contradiction.
	
	\medskip 
	
	Assume on the contrary that there exists a linear subspace $E \subseteq M_2(\C)$ such that $\mu_E(A)=h_t(A)$ for all $A \in M_2(\C)$. If $\dim(E) \leq 2$, then $\dim(E_\hp)= 4-\dim(E) \geq 2$. By Lemma \ref{lem_301}, $E_\hp$ has a rank one matrix and by Lemma \ref{lem_302}, there is a rank one matrix $A \in M_2(\C)$ with $\mu_E(A)=0$. Thus, $h_t(A)=0$ for a non-zero $A$ in $M_2(\C)$, which is a contradiction. If $\dim(E)=4$, then $\mu_E=h_t=\|.\|$, which is not possible as discussed above. Thus, $\dim(E)=3$. Let $U$ be a unitary in $M_2(\C)$. Following the similar arguments as in Theorem \ref{thm_numrigidI}, we have that 
	\begin{align}\label{eqn_muEUV}
		\mu_{E_U}(A)=\mu_E(U^*AU)=h_t(U^*AU)=h_t(A)=\mu_E(A)
	\end{align}
	for all $A \in M_2(\C)$ and unitary $U \in M_2(\C)$, where $E_U=\{Y \in M_2(\C) : U^*YU \in E\}$.
	As $\dim(E)=3$, we can write $E_\hp=\text{span}\{M\}$. If $M$ has rank one, then by Lemma \ref{lem_302}, there is a rank one matrix $A \in M_2(\C)$ with $\mu_E(A)=0$, which is not possible as $h_t(A)>0$. Now assume that $M$ has rank $2$. Then $M$ is invertible and we choose a unitary $U$ such that 
	\[
	UMU^*=\begin{pmatrix}
		\lm_1 & m_{12}\\
		0 & \lm_2 
	\end{pmatrix} \quad \text{and so,} \quad U\left(\lm_1^{-1}M\right)U^*=\begin{pmatrix}
		1 & x \\
		0 & y
	\end{pmatrix},
	\]
	where $x=\lm_1^{-1}m_{12}$ and $y=\lm_1^{-1}\lm_2$. Set $M_1=\lm_1^{-1}M$ and $N=UM_1U^*$. Clearly, $M_1 \in E_\hp$ and $E=\{X \in M_2(\C) : tr(M_1^*X)=0\}$. Also,
	$
	E_U=\{Y \in M_2(\C) : \text{tr}(M_1^*U^*YU) =0\}=\{Y \in M_2(\C) : \text{tr}(N^*Y) =0\}
	$
	and thus, $(E_U)_\hp=\text{span}\{N\}$. Moreover, we have by \eqref{eqn_muEUV} that 
	\begin{align}\label{eqn_muEUVI}
		\mu_{E_U}(A)=h_t(A) \quad \text{for all $A \in M_2(\C)$.}
	\end{align} 
	In particular, we must have 
	\begin{align}\label{eqn_E12_E21}
		\mu_{E_U}(E_{ii})=h_t(E_{ii})=1 \quad \text{and} \quad 	\mu_{E_U}(E_{12})=h_t(E_{12})=1-t_1-\frac{1}{2}t_2=h_t(E_{21})=\mu_{E_U}(E_{12})
	\end{align} 
	for $i=1,2$. Hence, $\inf\{\|X\|: X \in E_U, \det(I-E_{ii}X)=0\}=1$ for $i=1, 2$. Repeating the same arguments as in Theorem \ref{thm_numrigidI}, one can find $P=\diag(p_{11}, 1)$ and  $Q=\diag(1, q_{22})$ in $E_U$.	If $P, Q$ are linearly dependent, then it follows from the proof of Theorem \ref{thm_numrigidI} that $\mu_{E_U}(E_{12})=1$. By \eqref{eqn_muEUVI}, $h_t(E_{12})=1-t_1-t_2\slash 2=1$ and thus, $(t_1, t_2)=(0, 0)$ which is again a contradiction. Now suppose $P, Q$ are linearly independent. Then  $\text{span}\{P, Q\}=E_\diag \subseteq E_U$. In this case, we can write 
	\[
	E_U=\text{span}\left\{\begin{pmatrix}
		1 & 0\\
		0 & 0
	\end{pmatrix}, \begin{pmatrix}
		0 & 0\\
		0 & 1
	\end{pmatrix}, \begin{pmatrix}
		0 & \alpha\\
		\beta & 0
	\end{pmatrix}\right\} \ \ \text{for some $\alpha, \beta \in \C$}.
	\]
 We follow the similar arguments as in Proposition \ref{prop_304} depending on the values of $\alpha$ and $\beta$. 
	
	\medskip 
	
	\noindent \textbf{Case 1.} Let $\alpha=0$. Then $\beta \neq 0$, otherwise $\dim(E_U)=2$. Evidently, $E_U = \mathrm{span}\{E_{11}, E_{22}, E_{21}\}$. For $X=\begin{pmatrix}
		a & 0 \\
		c & b
	\end{pmatrix} \in E_U$, we have that $\det(I - E_{21}X) = 1$ and $\det(I-E_{12}X)=1-c$. Then  
	\[
	\inf\{\|X\|: X \in E_U, \det(I-E_{12}X)=0\}=\inf\left\{\left\|\begin{pmatrix}
		a & 0 \\
		1 & b
	\end{pmatrix}\right\|: a, b \in \C\right\}=1. 
	\] 
	Thus, $\mu_{E_U}(E_{12})=1$ and $\mu_{E_U}(E_{21})=0$, which is a contradiction to \eqref{eqn_E12_E21}.	
	\medskip 
	
	\noindent \textbf{Case 2.} Let $\beta=0$. Then $\alpha \neq 0$, otherwise $\dim(E_U)=2$. Evidently, $E_U = \mathrm{span}\{E_{11}, E_{22}, E_{12}\}$. For $X \in E_U$, $\det(I - E_{12}X) = 1$. Also, for $X=\begin{pmatrix}
		a & c \\
		0 & b
	\end{pmatrix} \in E_U$, $\det(I-E_{21}X)=1-c$. Then  
	\[
	\inf\{\|X\|: X \in E_U, \det(I-E_{21}X)=0\}=\inf\left\{\left\|\begin{pmatrix}
		a & 1 \\
		0 & b
	\end{pmatrix}\right\|: a, b \in \C\right\}=1. 
	\] 
	Thus, $\mu_{E_U}(E_{12})=0$ and $\mu_{E_U}(E_{21})=1$, which is again a contradiction to \eqref{eqn_E12_E21}.		
	\medskip 
	
	\noindent \textbf{Case 3.} Let $\alpha, \beta \ne 0$ with $|\alpha| \leq |\beta|$. We can simply write 
	\[
	E_U = \text{span}\left\{E_{11}, E_{22}, \begin{pmatrix}
		0 & x \\
		1 & 0
	\end{pmatrix}\right\}=\left\{\begin{pmatrix}
		a & \gamma \ x \\
		\gamma & b
	\end{pmatrix} : a, b, \gamma \in \C \right\}, \ \ \text{where $x=\alpha \slash \beta$}.
	\]
 For every $X=\begin{pmatrix}
		a & \gamma \ x \\
		\gamma & b
	\end{pmatrix} \in E_U$, $\det(I-E_{12}X)=1-\gamma$ and $\det(I-E_{21}X)=1-\gamma x$. Then
	\[
	1 \leq \inf\{\|X\|: X \in E_U, \det(I-E_{12}X)=0\}=\inf\left\{\left\|\begin{pmatrix}
		a & x \\
		1 & b
	\end{pmatrix}\right\| : a, b \in \C \right\} \leq \left\|\begin{pmatrix}
		0 & x \\
		1 & 0
	\end{pmatrix}\right\| \leq 1,
	\]
	where the last inequality holds as $|x| \leq 1$. Thus, $\mu_{E_U}(E_{12})=1$. Furthermore, we have 
	\[
	\frac{1}{|x|} \leq \inf\{\|X\|: X \in E_U, \det(I-E_{21}X)=0\}=\inf\left\{\left\|\begin{pmatrix}
		a & 1 \\
		1\slash x & b
	\end{pmatrix}\right\| : a, b \in \C \right\} \leq \left\|\begin{pmatrix}
		0 & 1 \\
		1\slash x & 0
	\end{pmatrix}\right\| = \frac{1}{|x|},
	\]
	where the last equality follows since $0<|x| \leq 1$. Thus, $\mu_{E_U}(E_{21})=|x|$ and by \eqref{eqn_E12_E21}, $|x|=1$. Some routine computation shows that
	\[
	(E_U)_\hp=\{Y \in M_2(\C): \text{tr}(X^*Y)=0 \ \text{for all} \ X \in E \}=\text{span}\left\{V\right\}, \quad \text{where} \ \ V=\begin{pmatrix}
		0 & 1 \\
		-\overline{x} & 0
	\end{pmatrix}.
	\]
	Since $|x|=1$, we have that $V$ is a unitary. By definition, $E_U=\{X \in M_2(\C) : tr(V^*X)=0\}$, and by Example \ref{cor_Munitary}, $\mu_{E_U}=\|.\|$. Thus, $h_t=\|.\|$, which is not possible as shown earlier. 
	\medskip 
	
	\noindent \textbf{Case 4.} Let $\alpha, \beta \ne 0$ with $|\beta| \leq |\alpha|$. We can simply write 
	\[
	E_U = \text{span}\left\{E_{11}, E_{22}, \begin{pmatrix}
		0 & 1 \\
		y & 0
	\end{pmatrix}\right\}=\left\{\begin{pmatrix}
		a & \gamma  \\
		\gamma \ y & b
	\end{pmatrix} : a, b, \gamma \in \C \right\},
	\]
	where $y=\beta\slash \alpha$. For every $X=\begin{pmatrix}
		a & \gamma  \\
		\gamma \ y & b
	\end{pmatrix} \in E_U$, we have that $\det(I-E_{21}X)=1-\gamma$. Then
	\[
	1 \leq \inf\{\|X\|: X \in E_U, \det(I-E_{21}X)=0\}=\inf\left\{\left\|\begin{pmatrix}
		a & 1 \\
		y & b
	\end{pmatrix}\right\| : a, b \in \C \right\} \leq \left\|\begin{pmatrix}
		0 & 1 \\
		y & 0
	\end{pmatrix}\right\| \leq 1.
	\]
	Thus, $\mu_{E_U}(E_{21})=1$. Also,  $\det(I-E_{12}X)=1-\gamma y$ and so, we have
	\[
	\frac{1}{|y|} \leq \inf\{\|X\|: X \in E_U, \det(I-E_{12}X)=0\}=\inf\left\{\left\|\begin{pmatrix}
		a & 1\slash y \\
		1 & b
	\end{pmatrix}\right\| : a, b \in \C \right\} \leq \left\|\begin{pmatrix}
		0 & 1\slash y \\
		1 & 0
	\end{pmatrix}\right\| = \frac{1}{|y|}.
	\]
	Thus, $\mu_{E_U}(E_{12})=|y|$ and by \eqref{eqn_E12_E21}, $|y|=1$. A simple calculation shows that
	\[
	(E_U)_\hp=\{Y \in M_2(\C): \text{tr}(X^*Y)=0 \ \text{for all} \ X \in E_U \}=\text{span}\left\{V_0\right\}, \quad \text{where} \ \ V_0=\begin{pmatrix}
		0 & -\overline{y}\\
		1 & 0
	\end{pmatrix}.
	\]
	Since $|y|=1$, we have that $V_0$ is a unitary. By definition, $E_U=\{X \in M_2(\C) : tr(V_0^*X)=0\}$. We have by Example \ref{cor_Munitary} that $\mu_{E_U}=\|.\|$. Thus, $h_t=\|.\|$, which is not possible as discussed earlier. In either case, we have a contradiction and so, the desired conclusion follows.
\end{proof} 

\vspace{0.3cm}

\noindent \textbf{Funding.} The first named author is supported in part by Core Research Grant with Award No. CRG/2023/005223 from Anusandhan National Research Foundation (ANRF) of Govt. of India. The second named author is supported via the IIT Bombay RDF Grant of the first named
author with Project Code RI/0115-10001427.


\begin{thebibliography}{9}
	
	\bibitem{Abouhajar}
	A. A. Abouhajar, M. C. White, N. J. Young, \textit{A Schwarz lemma for a domain related to $\mu$-synthesis}, J. Geom. Anal., 17 (2007), 717 -- 750. \smallskip 
	

\bibitem{AglerIV}	
J. Agler, Z. A. Lykova, N. J. Young, \textit{The complex geometry of a domain related to $\mu$-synthesis}, J. Math. Anal. Appl., 422 (2015), 508 -- 543.\smallskip 

	
	\bibitem{AglerYoung}	
	J. Agler, N. J. Young, \textit{A commutant lifting theorem for a domain in $\C^2$ and spectral interpolation}, J. Funct. Anal., 161 (1999), 452 -- 477. \smallskip
	
	 \bibitem{AglerI}
J. Agler, N. J. Young, \textit{The two-point spectral Nevanlinna-Pick problem}, Integral Equations Operator Theory, 37 (2000), 375 -- 385.\smallskip

\bibitem{Agler2004}
J. Agler, N. J. Young, \textit{The two-by-two spectral Nevanlinna-Pick problem}, Trans. Amer. Math. Soc., 356 (2004), 573 -- 585. \smallskip 
	
	\bibitem{AglerIII}
	J. Agler, N. J. Young, \textit{The hyperbolic geometry of the symmetrized bidisc}, J. Geom. Anal., 14 (2004), 375 -- 403. \smallskip 
	
	
	\bibitem{Agler2008}
	J. Agler, N. J. Young, \textit{The magic functions and automorphisms of a domain}, Complex Anal. Oper. Theory, 2 (2008), 383 -- 404. \smallskip 	
	
	
	\bibitem{Alsalhi}
	O. M. O. Alsalhi, Z. A. Lykova, \textit{Rational tetra-inner functions and the special variety of the tetrablock}, J. Math. Anal. Appl., 506 (2022), 125534, 52 pp.
	
	 \smallskip
	 
	 \bibitem{Ball}
J. A. Ball, I. Gohberg, L. Rodman, \textit{Interpolation of rational matrix functions}, OT45, Birkhauser Verlag, 1990. \smallskip 
	 
	 \bibitem{Bercovici}
H. Bercovici, C. Foias, A. Tannenbaum, \textit{A spectral commutant lifting theorem}, Trans. Amer. Math. Soc., 325 (1991), 741 -- 763. 

\smallskip

	\bibitem{Tirtha}
T. Bhattacharyya, \textit{The tetrablock as a spectral set}, Indiana Univ. Math. J., 63 (2014), 1601 -- 1629. \smallskip 

	
	\bibitem{Tirtha_Pal}
	T. Bhattacharyya, S. Pal, S. S. Roy, \textit{Dilations of $\Gamma$-contractions by solving operator equations}, Adv. Math., 230 (2012), 577 -- 606. \smallskip 
	
		\bibitem{Hexa_Su}
	E. Bi, Z. Shaaban, G. Su, \textit{Rigidity of proper holomorphic self-mappings of the hexablock}, arXiv: 2507.16176.\smallskip
	
	\bibitem{Hexablock} 
	I. Biswas, S. Pal, N. Tomar, \textit{The Hexablock: a domain associated with the $\mu$-synthesis in $M_2(\mathbb{C})$}, arXiv: 2506.15149. \smallskip 
	
	
	\bibitem{Costara2004}		
	C. Costara, \textit{The symmetrized bidisc and Lempert's theorem}, Bull. London Math. Soc., 36 (2004), 656 -- 662. \smallskip
	
	\bibitem{Costara2005_II}
C. Costara, \textit{On the spectral Nevanlinna-Pick problem}, Studia Math., 170 (2005), 23 -- 55.\smallskip 
	
	\bibitem{Costara2005}		
	C. Costara, \textit{The $2 \times 2$ spectral Nevanlinna-Pick problem}, J. London Math. Soc., 71 (2005), 684 -- 702. \smallskip 
	
	\bibitem{Doyle}
	J. Doyle, \textit{Analysis of feedback systems with structured uncertainties}, IEE Proc. Control Theory Appl., 129 (1982), 242 -- 250. \smallskip 
	
	\bibitem{DoyleII}
	J. Doyle, G. Stein, \textit{Multivariable feedback design: concepts for a classical\slash modern synthesis}, IEEE Transactions on Automatic Control, 26 (1981), 4 -- 16. \smallskip 
	
	\bibitem{Foias_Frazho}
C. Foias, A. E. Frazho, \textit{The commutant lifting approach to interpolation problem}, Birkh\"{a}user, Berlin, 1990. \smallskip
	
		\bibitem{JindalII}
	A. Jindal, P. Kumar, \textit{Operator theory on the pentablock},
	J. Math. Anal. Appl., 540 (2024), no. 1, Paper No. 128589, 17 pp.		\smallskip
	
	\bibitem{Kosi}
	L. Kosi\'{n}ski, \textit{Geometry of quasi-circular domains and applications to tetrablock}, Proc. Amer. Math. Soc., 139 (2011), 559 -- 569.\smallskip
	
	
	\bibitem{KosinskiII}
	L. Kosi\'{n}ski, \textit{The group of automorphisms of the pentablock}, Complex Anal. Oper. Theory, 9 (2015), 1349 -- 1359. \smallskip
	
	
	\bibitem{zwonek1} 
	L. Kosi\'nski, W. Zwonek, \textit{Nevanlinna-Pick problem and uniqueness of left inverses in convex domains,
		symmetrized bidisc and tetrablock}, J. Geom. Anal., 26 (2016), 1863 -- 1890. \smallskip

	\bibitem{Talha}
	 T. Mushtaq, P. Seiler, M. S. Hemati, \textit{Exact solution for the rank-one structured singular value with repeated complex full-block uncertainty}, Automatica J. IFAC, 167 (2024), Paper No. 111717, 4 pp.
	 
	  \smallskip 
	  
	 \bibitem{NikolovIII}
	N. Nikolov, P. Pflug, P. J. Thomas, \textit{Spectral Nevanlinna-Pick and Carath\'{e}odory-Fej\'{e}r problems for $n \leq 3$}, Indiana Univ. Math. J., 60 (2011), 883 -- 893. \smallskip
	
	\bibitem{Packard}
	A. Packard, J. Doyle, \textit{The complex structured singular value},  Automatica J. IFAC, 29 (1993), 71 -- 109. \smallskip	
	
	\bibitem{sourav14} 
	S. Pal, \textit{From Stinespring dilation to Sz.-Nagy
		dilation on the symmetrized bidisc and operator models},
	New York J. Math., 20 (2014), 645 -- 664. \smallskip	 
	 
	\bibitem{pal-shalit} 
	S. Pal, O. M. Shalit, \textit{Spectral sets and distinguished varieties in the symmetrized bidisc}, J.	Funct. Anal., 266 (2014), 5779 -- 5800. \smallskip	
	
	\bibitem{PalN}
	S. Pal, N. Tomar, \textit{Operators associated with the pentablock and their relations with biball and symmetrized bidisc}, arXiv: 2309.15080. \smallskip
	
	\bibitem{Pal_Hexa}
	S. Pal, N. Tomar, \textit{Operators associated with the hexablock}, arXiv: 2507.14589.
	\smallskip
	
		\bibitem{PZ}
	P. Pflug,  W. Zwonek, \textit{Description of all complex geodesics in the symmetrized bidisc}, Bull. London Math. Soc., 37 (2005), 575 -- 584.\smallskip

	\bibitem{Guicong}
G. Su, \textit{Geometric properties of the pentablock}, Complex Anal. Oper. Theory, 14 (2020), Paper No. 44, 14 pp.
\smallskip

	
	\bibitem{GuicongII}
	G. Su, Z. Tu, L. Wang, \textit{Rigidity of proper holomorphic self-mappings of the pentablock}, J. Math. Anal. Appl., 424 (2015), 460 -- 469.
	\smallskip
	
			
	
	\bibitem{Young}
	N. J. Young, \textit{The automorphism group of the tetrablock}, J. Lond. Math. Soc., 77 (2008), 757 -- 770.\smallskip
	
	\bibitem{Zwonek}
	W. Zwonek, \textit{Geometric properties of the tetrablock}, Arch. Math., 100 (2013), 159 -- 165.
\end{thebibliography}
\end{document}